\documentclass[a4paper,11pt,fleqn]{article} %  fleqn equations center to the left
\usepackage{amsmath}
\usepackage{amsthm}
\usepackage{amssymb}
\usepackage[T1]{fontenc} % for accent
\usepackage[utf8]{inputenc} % unicode
\usepackage{mathrsfs} %pour mathscr
\usepackage{euscript}
\usepackage{mathtools}
\usepackage[dvipsnames]{xcolor} % for colors
\usepackage[shortlabels]{enumitem} % give ability to customize list: itemize, enumerate, description.
\usepackage{bm} % bold
\usepackage{charter} % serif font type
\usepackage[scaled=0.95]{helvet} % sans serif font type
\usepackage[normalem]{ulem} % underlining for \emph, but normalem disable it. command exple: \(u)uline, \uwave, \sout(cross word), \dotuline, \dashuline  ....
\usepackage[misc]{ifsym} % for letter sign \Letter
\usepackage[toc,titletoc,title]{appendix}

\usepackage{caption}
\usepackage{subcaption}
\definecolor{labelkey}{rgb}{0,0.08,0.45}
\definecolor{refkey}{rgb}{0,0.6,0.0}
\definecolor{Brown}{rgb}{0.45,0.0,0.05}
\definecolor{dgreen}{rgb}{0.00,0.49,0.00}
\definecolor{dblue}{rgb}{0,0.08,0.75}
\definecolor{overleaf_mostly_pure_blue}{rgb}{0, .502, 1} % #0080ff
\definecolor{overleaf_dark_moderate_blue}{rgb}{.353, .361, .678} % #5A5CAD
\definecolor{overleaf_dark_moderate_cyan}{rgb}{.247, .498, .498} % #3F7F7F
\definecolor{colorhexa_dark_blue}{rgb}{0, .4, .6} % #006699
\definecolor{colorhexa_dark_orange}{rgb}{.6, .2, 0} % #993300
\RequirePackage[colorlinks,hyperindex]{hyperref} %click pdf
\hypersetup{linktocpage=true,citecolor=dblue,linkcolor=dgreen}
\newtheorem{theorem}{Theorem}[section]
\newtheorem{corollary}[theorem]{Corollary}
\newtheorem{lemma}[theorem]{Lemma}

\newtheorem{proposition}[theorem]{Proposition}

\newtheorem{assumptions}[theorem]{Assumptions}
\theoremstyle{definition}

\newtheorem{algorithm}[theorem]{Algorithm}
\newtheorem{remark}[theorem]{Remark}

\numberwithin{equation}{section}

\providecommand{\norm}[1]{\lVert#1\rVert}
\providecommand{\scalarp}[1]{\langle#1\rangle}

\newcommand{\Fk}{\ensuremath{\mathfrak F}}
\newcommand{\EE}{\ensuremath{\mathsf E}}

\newcommand{\PP}{\ensuremath{\mathsf P}}
\newcommand{\R}{\ensuremath \mathbb{R}}
\newcommand{\HH}{\ensuremath \mathsf{H}}

\newcommand{\HHs}{\ensuremath \mathsf{H}}

\newcommand{\bxs}{\ensuremath \bm{\mathsf{x}}}
\newcommand{\bys}{\ensuremath \bm{\mathsf{y}}}
\newcommand{\bas}{\ensuremath \bm{\mathsf{a}}}

\newcommand{\xx}{\ensuremath \mathsf{x}}

\newcommand{\zz}{\ensuremath \mathsf{z}}

\newcommand{\yy}{\ensuremath \mathsf{y}}

\newcommand{\N}{\ensuremath \mathbb{N}}

\newcommand{\prox}{\ensuremath{\text{\textsf{prox}}}}

\DeclareMathOperator*{\argmin}{\text{\rm{argmin}}}

\DeclareMathOperator*{\dist}{\text{\rm{dist}}}
\newcommand{\minimize}[2]{\ensuremath{\underset{\substack{{#1}}}{\text{\textrm{minimize}}}\;\;#2 }}

\newcommand{\bSS}{\ensuremath{\mathsf{S}}}

\newcommand{\bxx}{\ensuremath \bm{\xx}}

\newcommand{\byy}{\ensuremath \bm{\yy}}

\newcommand{\bphi}{\ensuremath \bm{\phi}}
\newcommand{\bx}{\ensuremath \bm{x}}
\newcommand{\by}{\ensuremath \bm{y}}
\newcommand{\bv}{\ensuremath \bm{v}}

\newcommand{\bg}{\ensuremath \bm{e}}

\newcommand{\bu}{\ensuremath \bm{u}}
\newcommand{\bb}{\ensuremath \bm{b}}

\newcommand{\bw}{\ensuremath \bm{w}}

\newcommand{\bvarphi}{\ensuremath{{\varphi}}}

\newcommand{\bbar}{\bm \bar}

\newcommand{\grandO}[1]{O\mathopen{}\left(#1\right)}

\begin{document}

\title{ {\sffamily Variance reduction techniques for stochastic proximal point algorithms}}
\author{Cheik Traor\'e$\,^{\textrm{\Letter}}\,$\thanks{Corresponding author: Malga Center, DIMA, Universit\`a degli Studi di Genova, Genova, Italy (traore@dima.unige.it). {\scriptsize This project has received funding from the European Union’s Horizon 2020 research and innovation programme under the Marie Skłodowska-Curie grant agreement No 861137.}}\,,\ \
%%%%%%%%%%%%%%%%%%%%%%%%%%%%%%%%%%%%%%%%%
Vassilis Apidopoulos\thanks{Malga Center, Dibris, Universit\`a degli Studi di Genova, Genova, Italy (vassilis.apid@gmail.com).}\,,\ \
%%%%%%%%%%%%%%%%%%%%%%%%%%%
Saverio Salzo\thanks{Sapienza Universit\`a di Roma, Roma, Italy (salzo@diag.uniroma1.it).}\ \
%%%%%%%%%%%%%%%%%%%%%%%%%%%%%%%%
and Silvia Villa\thanks{Malga Center, DIMA, Universit\`a degli Studi di Genova, Genova, Italy (silvia.villa@unige.it). {\scriptsize She acknowledges the financial support of the European
Research Council (grant SLING 819789), the AFOSR projects FA9550-17-1-0390, FA8655-22-1-7034, and BAAAFRL-
AFOSR-2016-0007 (European Office of Aerospace Research and Development), and the EU H2020-MSCA-RISE project NoMADS - DLV-777826.}}
        }
\date{}
\maketitle

\begin{abstract}
% {\color{red}In the context of finite sums minimization, variance reduction techniques are widely used to improve the performance of state-of-the-art stochastic gradient methods. Their practical impact is clear, as well as their theoretical properties.  Stochastic proximal point algorithms have been studied
% as an alternative to stochastic gradient algorithms since they are more stable with respect to the choice of the step size but a proper  variance-reduced 
% version is missing. 
% In this work, we propose the first study of variance reduction techniques for stochastic proximal point algorithms. We introduce a stochastic 
% proximal version of SVRG, SAGA, and some of their variants for smooth and convex functions. We provide
% several convergence results for the iterates and the objective function values. In addition, under the Polyak-Łojasiewicz (PL) condition, 
% we obtain linear convergence rates for the iterates and the function values. Our numerical experiments demonstrate the advantages of the proximal 
% variance reduction methods over their gradient counterparts, especially about the stability with respect to the choice of the step size.}
\noindent
 In the context of finite sums minimization, variance reduction techniques are widely used to improve the performance of state-of-the-art stochastic gradient methods. Their practical impact is clear, as well as their theoretical properties. Stochastic proximal point algorithms have been studied as an alternative to stochastic gradient algorithms since they are more stable with respect to the choice of the step size. However, their variance-reduced versions are not as well studied as the gradient ones. In this work, we propose the first unified study of variance reduction techniques for stochastic proximal point algorithms. We introduce a generic stochastic proximal-based algorithm that can be specified to give the proximal version of SVRG, SAGA, and some of their variants. 
 % The main goal of this paper is to prove faster rates with respect to the ones of the vanilla stochastic proximal point algorithm. 
For this algorithm, in the smooth setting, we provide several  convergence rates for the iterates and the objective function values, which are faster than those of the vanilla stochastic proximal point algorithm. More specifically, for convex functions, we prove a sublinear convergence rate of $O(1/k)$. In addition, under the Polyak-Łojasiewicz (PL) condition, we obtain linear convergence rates. Finally, our numerical experiments demonstrate the advantages of the proximal variance reduction methods over their gradient counterparts in terms of the stability with respect to the choice of the step size in most cases, especially for difficult problems.
\end{abstract}

\vspace{2ex}

\vspace{2ex}
% \noindent
% \text{Communicated by Sonia Cafieri.}
% \vfill
\begin{section}{Introduction}
The objective of the paper is to solve the following finite-sum optimization problem
\begin{align}
  \label{prob:main}
    \minimize{ \bxs \in \HHs}{F(\bxs)} = \frac{1}{n}\sum_{i = 1}^{n}f_i(\bxs),
\end{align}
where  $\HHs$ is a separable Hilbert space and for all $i \in \{1,2, \cdots, n \}$, $f_i:\HHs\to \R$.

Several problems can be expressed as in \eqref{prob:main}. The most popular example is the Empirical Risk Minimization (ERM) problem in machine learning \cite[Section 2.2]{shalev2014understanding}. In that setting, $n$ is the number of data points, $\bxs\in\mathbb{R}^d$ includes the parameters of a machine learning model (linear functions, neural networks, etc.), and the function $f_i$ is the loss of the model $\bxs$ at the $i$-th data point.
Due to the large scale of data points used in machine learning/deep learning, leveraging gradient descent (GD) for the problem \eqref{prob:main} can be excessively costly both in terms of computational power and storage. To overcome these issues, several stochastic variants of gradient descent have been proposed in recent years.
%%%%%%%%%%%%%%%%%%%%%%%%%%%%%%%%%%%%%%%%%%%%%%%%%%%%%%%%%%%%%%%%%%%%%%%%%%%%%%%
\paragraph{Stochastic Gradient Descent.} The most common version of stochastic approximation \cite{robbins1951stochastic} applied to \eqref{prob:main} is that where at each step the full gradient is replaced by $\nabla f_i$, the gradient of a function $f_i$, with $i$ sampled uniformly among $\{1, \cdots, n\}$. This procedure yields stochastic gradient descent (SGD), often referred to as incremental SGD. In its vanilla version or modified ones (AdaGrad \cite{duchi2011adaptive}, ADAM \cite{kingma2014adam}, etc.), SGD is ubiquitous in modern machine learning and deep learning. Stochastic approximation  \cite{robbins1951stochastic} provides the appropriate framework to study the theoretical properties of the SGD algorithm, which are nowadays well understood \cite{bottou1999line,bottou2010large,moulines2011non,khaled2020better}. Theoretical analysis shows that SGD has a worse convergence rate compared to its deterministic counterpart GD. Indeed, GD exhibits a convergence rate for the function values ranging from $O(1/k)$ for convex functions to $O(\exp\{-C_{\mu} k\})$ for  $\mu$-strongly convex functions, while for SGD the convergence rates of the function values vary from $O(1/\sqrt{k})$ to $O(1/k)$. In addition, convergence of SGD is guaranteed if the step size sequence is square summable, but not summable. In practice, this requirement is not really meaningful, and the appropriate choice of the step size is one of the major issues in SGD implementations. In particular, if the initial step size is too large, the SGD blows up even if the sequence of step sizes satisfies the suitable decrease requirement; see, e.g., \cite{moulines2011non}. Therefore, the step size needs to be tuned by hand, and for solving problem \eqref{prob:main},  this is typically time consuming.

\paragraph{Stochastic Proximal Point Algorithm.} Whenever the computation of the proximity operator of $f_i$, $\prox_{f_i}$ (for a definition see notation paragraph \ref{notation subsection}), is tractable, an alternative to SGD is the stochastic proximal point algorithm (SPPA). Instead of the gradient $\nabla f_i$, the proximity operator of a $f_i$, chosen randomly, is used in each iteration. Recent works, in particular \cite{ryu2014stochastic,asi2019stochastic, kim2022convergence}, showed that SPPA is more robust to the choice of step size with respect to SGD. In addition, the convergence rates are the same as those of SGD, in various settings \cite{bertsekas2011incremental,patrascu2017nonasymptotic,asi2019stochastic}, possibly including a momentum term \cite{kim2022convergence,toulis2016towards, toulis2017asymptotic, toulis2015proximal}. 
%%%%%%%%%%%%%%%%%%%%%%%%%%%%%%%%%%%%%%%%%%%%%%%%%%
%%%%%%%%%%%%%%%%%%%
\paragraph{Variance reduction methods.} As  observed above, the convergence rates for SGD are worse than those of their deterministic counterparts. This is due to the non-vanishing variance of the stochastic estimator of the true gradient. In order to circumvent this issue, starting from SVRG \cite{johnson2013accelerating,allen2016improved}, a new wave of SGD algorithms was developed with the aim of reducing the variance and recovering standard GD rates with constant step size. Different methods were developed and all share a $O(1/k)$ and a $O(e^{-C_{\mu}k})$ convergence rate for function values for convex and $\mu$-strongly convex objective functions, respectively. In the convex case, the convergence is ergodic. Apart from SVRG, the other methods share the idea of reducing the variance by aggregating different stochastic estimates of the gradient. Among them we mention  SAG \cite{schmidt2017minimizing} and SAGA \cite{defazio2014saga}. In subsequent years, a plethora of papers have appeared on variance reduction techniques; see, for example, \cite{kovalev2020don, fang2018spider, wang2019spiderboost, nguyen2017sarah}. The paper \cite{gorbunov2020unified} provided a unified study of variance reduction techniques for SGD that encompasses many of them. 
The latter work \cite{gorbunov2020unified} has inspired our unified study of variance reduction for SPPA. 
% To the best of our knowledge there is neither a variance-reduced version of SPPA nor a unified study in the literature. 
%%%%%%%%%%%%%%%%%%%%%%%%%%%%%%%%%%%%%%%
{\paragraph{Variance-reduced Stochastic Proximal Point Algorithms.} The application of variance reduction techniques to SPPA is very recent and limited: the existing methods are Point-SAGA in \cite{defazio2016simple}, the proximal  version of L-SVRG \cite{kovalev2020don}, and  SNSPP proposed in \cite{milzarek2022semismooth}. All existing convergence results are provided in the smooth case, except for Point-SAGA in \cite{defazio2016simple}, where an ergodic and sublinear convergence rate with a constant step size is provided for nonsmooth strongly convex functions. }

%%%%%%%%%%%%%%%%%%%%%%%%%%%%%%%%%%%%%%%%%%%%%%%%%%%%%%%%%%%%%%%%%%%%%%%%%%%%%%
\paragraph{Contributions.}
Our contribution can be summarized as follows:
\begin{itemize}
    % \item  Assuming that the functions $f_i$ are smooth, we propose a unified variance reduction technique for stochastic proximal point algorithm (SPPA). We devise a unified analysis that extends several variance reduction techniques used for SGD to SPPA, as listed in Section \ref{sect:appli}, with improved rates over SPPA. In particular, we prove a sublinear convergence rate $\mathcal{O}(1/k)$ of the function values for convex functions, compared to $\mathcal{O}(1/\sqrt{k})$ for SPPA. The analysis in the convex case is new in the literature. Assuming additionally that the objective function $F$ satisfies the Polyak-Łojasiewicz (PL) condition, we prove linear convergence rate both for the iterates and the function values, better than the rate of $\mathcal{O}(1/k)$ for SPP \cite[Proposition 5.3]{asi2019stochastic}. The PL condition on $F$ is less strong than the strong convexity of $F$ or even $f_i$ used in the related previous work. Finally, we show that these results are achieved for constant step sizes.
    \item  Assuming that the functions $f_i$ are smooth, we propose a unified variance reduction technique for stochastic proximal point algorithm (SPPA). We devise a unified analysis that extends several variance reduction techniques used for SGD to SPPA, as listed in Section \ref{sect:appli}, with improved rates over SPPA. In particular, we prove a sublinear convergence rate $\mathcal{O}(1/k)$ of the function values for convex functions. The analysis in the convex case is new in the literature. Assuming additionally that the objective function $F$ satisfies the Polyak-Łojasiewicz (PL) condition, we prove linear convergence rate both for the iterates and the function values. The PL condition on $F$ is less strong than the strong convexity of $F$ or even $f_i$ used in the related previous work. Finally, we show that these results are achieved for constant step sizes.
    \item As a byproduct, we derive and analyze some stochastic variance-reduced proximal point algorithms, in analogy to SVRG \cite{johnson2013accelerating}, SAGA \cite{defazio2014saga} and L-SVRG \cite{kovalev2020don}.
    \item The experiments show that, in most cases and especially for difficult problems, the proposed methods are more robust to the step sizes and converge with larger step sizes, while retaining at least the same speed of convergence as their gradient counterparts. This generalizes the advantages of SPPA over SGD (see \cite{asi2019stochastic, kim2022convergence}) to variance reduction settings.
\end{itemize}

\paragraph{Organization.} The rest of the paper is organized as follows: In Section \ref{sect:algo}, we present our generic algorithm and the assumptions we will need in subsequent sections. In Section \ref{sect:results}, we show the results pertaining to that algorithm. Then, in Section \ref{sect:appli}, we specialize the general results to particular variance reduction algorithms.  Section \ref{sect:experiments} collects our numerical experiments. Proofs of auxiliary results can be found in Appendix \ref{app:proof}. 
\end{section}
%%%%%%%%%%%%%%%%%%%%%%%%%%%%%%%%%%%%%%%%%%%%%%%%%%%%%%%%%%%%%%%%%%%%%%%%%%%%%%%%%%%%%%%
\begin{section}{Algorithm and assumptions}\label{sect:algo}
\begin{subsection}{Notation}\label{notation subsection}
We first introduce the notation and recall a few basic notions that will be needed throughout the
paper. We denote by $\N$ the set of natural numbers (including zero) and by $\R_+ = \left[0,+\infty\right[$ the set of positive real numbers.
For every integer $\ell \geq 1$, we define $[\ell] = \{1, \dots, \ell\}$.  $\HHs$ is a Hilbert space endowed with  scalar product  $\langle\cdot, \cdot\rangle$ and induced norm $\|\cdot\|$. If $\bSS \subset \HHs$ is convex and closed and $\bxs \in \HHs$, we set $\mathrm{dist}(\bxs,\bSS) = \inf_{\bm{\zz} \in \bSS} \norm{\bxx - \bm{\zz}}$. The projection of $\bxs$ onto $\bSS$ is denoted by $P_{\bSS}(\bxs)$. 

Bold default font is used for random variables taking values in $\HH$, while bold sans serif font is used for their realizations or deterministic variables in $\HH$. 
The probability space underlying random variables is denoted by $(\Omega, \mathfrak{A}, \PP)$. For every random variable $\bx$, $\EE[\bx]$ denotes its expectation, while if $\Fk \subset \mathfrak{A}$ is a sub $\sigma$-algebra we denote by $\EE[\bx\,\vert\, \Fk]$ the conditional expectation of $\bx$ given $\Fk$. Also, $\sigma(\by)$ represents the $\sigma$-algebra generated by the random variable $\by$.

Let $\bvarphi\colon \HHs \to \R$ be a function. The set of minimizers of $\bvarphi$ is $\argmin \bvarphi 
= \{\bxs \in \HHs \,\vert\, \bvarphi(\bxs) = \inf \bvarphi\}$. If $\inf \bvarphi$ is finite, it is represented by $\bvarphi_*$. When $\bvarphi$ is differentiable $\nabla \bvarphi$ denotes the gradient of $\bvarphi$. We recall that the proximity operator of $\varphi$ is defined as $\prox_{\varphi}(\bxs) = \argmin_{\bys \in \HHs} \varphi(\bys) + \frac{1}{2} \norm{\bys-\bxs}^2$.

In this work, $\ell^1$ represents the space of sequences which norms are summable and $\ell^2$ the space of sequences which norms are square summable.
\end{subsection}
%%%%%%%%%%%%%%%%%%%%%%%%%%%%%%%%%%%%%%%%%%%%%%%%%%%%%%%%%%%%%%%%%%%%%%%%%%%%%%
\begin{subsection}{Algorithm}
In this paragraph, we describe a generic method for solving problem \eqref{prob:main}, based on the stochastic proximal point algorithm.
\begin{algorithm}\label{algo:unified}
Let $(\bg^k)_{k \in \N}$ be a sequence of random vectors in $\HHs$ and let $(i_k)_{k \in \N}$ be a sequence of i.i.d. random variables uniformly distributed on $\{1,\dots, n\}$, so that $i_k$ is independent of $\bg^0$, \dots, $\bg^{k-1}$. Let $\alpha >0$ and set the initial point $\bx^0 \equiv \bxs^0 \in \HHs$. Then define
\begin{equation*}
\begin{array}{l}
\text{for}\;k=0,1,\ldots\\
\left\lfloor
\begin{array}{l}
\bx^{k+1} = \prox_{\alpha f_{i_{k}}}\big(\bx^k  + \alpha\bg^k \big). 
\end{array}
\right.
\end{array}
\end{equation*}
\end{algorithm}

Algorithm \ref{algo:unified} is a stochastic proximal point method including an additional term $\bg^k$, that will be suitably chosen to reduce the variance. As we shall see in Section \ref{sect:appli}, depending on the specific algorithm (SPPA, SVRP, L-SVRP, SAPA), $\bg^k$ may be defined in various ways. Note that $\bx^{k+1}$ is a random vector depending on $\bx^k, i_k$, and $\bg^k$. 
% {\color{red}Note that by definition of the proximal operator, assuming $f_{i_k}$ is differentiable and setting $\bw^k = \nabla f_{i_k}(\bx^{k+1}) - \bg^k$, the update in Algorithm \ref{algo:unified} can be also rewritten as:
% \begin{equation*}%\label{eq:20220915c}
%     \bx^{k+1}=\bx^k-\alpha\big[\nabla f_{i_k} (\bx^{k+1}) - \bg^k \big] = \bx^k -\alpha \bw^k.
% \end{equation*}
% }
{
By definition of the proximal operator, we derive from Algorithm~\ref{algo:unified}
\begin{align*}
    \bx^{k+1} &= \prox_{\alpha f_{i_{k}}}\big(\bx^k  + \alpha\bg^k \big)\\
    &= \argmin_{\bx \in \HH} f_{i_{k}}(\bx) + \frac{1}{2\alpha} \|\bx^k  + \alpha\bg^k - \bx\|^2.
\end{align*}
By the optimality condition and assuming $f_{i_k}$ is differentiable, we have
\begin{align}\label{eq:20240404a}
    0 = \nabla f_{i_{k}}(\bx^{k+1}) + \frac{1}{\alpha}\left(\bx^{k+1} - \bx^{k} - \alpha\bg^k\right).
\end{align}
Let $\bw^k = \nabla f_{i_k}(\bx^{k+1}) - \bg^k$. Thanks to \eqref{eq:20240404a}, the update in Algorithm \ref{algo:unified} can be rewritten as
\begin{equation}\label{eq:20220915c}
    \bx^{k+1}=\bx^k-\alpha\big[\nabla f_{i_k} (\bx^{k+1}) - \bg^k \big] = \bx^k -\alpha \bw^k.
\end{equation}

Keeping in mind that actually $\bw^k$ depends on $\nabla f_{i_k} (\bx^{k+1})$,  Equation \eqref{eq:20220915c} shows that Algorithm \ref{algo:unified} can be seen as an implicit stochastic gradient method, in contrast to an explicit one, where 
$\bw^k$ is replaced by
\begin{equation}\label{defvk}
    \bv^k \coloneqq \nabla f_{i_k}(\bx^{k}) - \bg^k.
\end{equation}
This point of view has been exploited in \cite[Appendix A]{gorbunov2020unified}, to provide a unified theory for variance reduction techniques for SGD.}
\begin{remark}
{ Due to the dependence of $\bx^{k+1}$ on $i_k$, in general $\EE[\nabla f_{i_k}(\bx^{k+1})\,\vert\,\bx^k] \neq \nabla F(\bx^{k+1})$ and $\EE[f_{i_k}(\bx^{k+1})\,\vert\,\bx^k] \neq F (\bx^{k+1})$, making the analysis of Algorithm \ref{algo:unified} tricky. We circumvent that problem using $\bv^k = \nabla f_{i_k}(\bx^{k}) - \bg^k$ as an auxiliary variable; see \eqref{eq:20220915d}. This is helpful since $\bx^k$ is independent of $i_k$. Therefore, even though $\bv^k$ does not appear explicitly in Algorithm \ref{algo:unified}, it is still relevant in the associated analysis of the convergence bounds, and, for those derivations, some assumptions will be required on $\bv^k$.}
\end{remark}

\begin{remark}{ Variance reduction techniques have been already connected to the proximal gradient algorithm to solve structured problems of the form $\sum_i f_i+g$, where $f_i$ are assumed to be smooth and $g$ is just prox friendly. Indeed, this is the  model corresponding to regularized empirical risk minimization \cite{shalev2014understanding}. For this objective functions the stochastic proximal gradient algorithm is as follows
\[
x_{k+1}=\prox_{\gamma_k g} (x_k-\gamma_k \nabla f_{i_k}(x_k)).
\]
As can be seen from the definition, no sum structure is assumed on the  function $g$ which is the one activated through the proximity operator. On the contrary, stochasticity arises from the computation of the gradient of $\sum_i f_i$ and variance reduction techniques can be exploited at this level. In this paper we tackle the same problem, in the special case where  $g=0$, but we analyze a different algorithm, where the functions $f_i$'s are activated through their proximity operator instead of their gradient. The addition of a regularizer $g$ can be considered, but at the moment differentiability is needed for our analysis. 
}
    
\end{remark}
\end{subsection}
%%%%%%%%%%%%%%%%%%%%%%%%%%%%%%%%%%%%%%%%%%%%%%%%%%%%%%%%%%%%%%%%%%%
\begin{subsection}{Assumptions}
The first assumptions are made on the functions $f_i:\HHs\to \R$, $i \in [n]$, as well as on the objective function $F:\HHs\to \R$.
\begin{assumptions}\label{ass:funtion} \
\begin{enumerate}[label=\rm(A.\roman*),leftmargin=*]
    \item \label{ass:a1} $\argmin F \neq \varnothing$.
    \item \label{ass:a11} For all $i \in [n]$, $f_i$ is convex and
     $L$-smooth, i.e., differentiable and such that
    \begin{equation*}
        (\forall \bxs,\bys \in \HHs) \quad \|\nabla f_i(\bxs) - \nabla f_i(\bys)\| \leq L\|\bxs - \bys\|
    \end{equation*}
    for some $L>0$.
    As a consequence, $F$ is convex and $L$-smooth.
    \item \label{ass:a12} $F$ satisfies the PL condition with constant $\mu>0$, i.e.,
\begin{equation}\label{PL condition}
    (\forall \bxs \in \HHs) \quad F(\bxs) - F_* \leq \frac{1}{2\mu} \|\nabla F(\bxs)\|^2,
\end{equation}
which is equivalent to the following quadratic growth condition when $F$ is convex
\begin{equation}\label{eq:mu_strongly_quasi_convex}
        (\forall \bxs \in \HHs) \quad\frac{\mu}{2}\dist(\bxs, \argmin F)^2 \leq F(\bxs) - F_*.
\end{equation}
\end{enumerate}
\end{assumptions}

\ref{ass:a1} and \ref{ass:a11} constitute the common assumptions that we use for all the convergence results presented in Sections \ref{sect:results} and \ref{sect:appli}. Assumption \ref{ass:a12} is often called \emph{Polyak-{\L}ojasiewicz condition} and was introduced in \cite{lojasiewicz1963propriete} (see also \cite{polyak1963gradient}) and is closely connected with the \emph{quadratic growth condition} \eqref{eq:mu_strongly_quasi_convex} (they are  equivalent in the convex setting, see e.g. \cite{bolte2017error}). Conditions \eqref{PL condition} and \eqref{eq:mu_strongly_quasi_convex} are both relaxations of the strong convexity property and are powerful key tools in establishing linear convergence for many iterative schemes, both in the convex \cite{karimi2016linear,bolte2017error,garrigos2022convergence,drusvyatskiy2018error} and the non-convex setting \cite{polyak1963gradient,polyak1964some,bolte2007lojasiewicz,begout2015damped,apidopoulos2022convergence}.

In the same fashion, in this work, Assumption \ref{ass:a12} will be used in order to deduce linear convergence rates in terms of objective function values for the sequence generated by Algorithm \ref{algo:unified}.
\begin{assumptions}\label{ass:variance}
Let, for all $k \in \N$, $\bv^k := \nabla f_{i_k}(\bx^{k}) - \bg^k$. Then
there exist non-negative real numbers $A, B, C \in \R_+$ and $\rho \in [0,1]$, and a non-positive real-valued random variable $D$ such that, for every $k \in \N$,
\begin{enumerate}[label=\rm(B.\roman*),leftmargin=*]
\item \label{ass:b1} $\EE[\bg^k\,\vert\, \Fk_k] = 0$\; a.s.,
    \item \label{ass:b2} $ \EE\left[\norm{\bv^k}^2 \,\vert\, \Fk_k\right] \le 2 A (F(\bx^k) - F_*) + B\sigma_k^2 + D$\; a.s.,
    \item \label{ass:b3} 
    $\EE\left[\sigma_{k+1}^2\right] \le (1-\rho) \EE\left[\sigma_k^2\right] + 2C\EE[F(\bx^k) - F_*]$,
\end{enumerate}
where $\sigma_k$ is a real-valued random variable,
$(\Fk_k)_{k \in \N}$ is a sequence of $\sigma$-algebras such that, $\forall k \in \N$, 
$\Fk_k \subset \Fk_{k+1} \subset \mathfrak{A}$,
$i_{k-1}$ and $x_k$ are $\Fk_k$-measurables, and 
$i_k$ is independent of $\Fk_k$.
\end{assumptions}
Assumption \ref{ass:b1} ensures that $\EE[\bv^k \,\vert\, \Fk_k] = \EE[\nabla f_{i_k} (\bx^k)\,\vert\, \Fk_k] = \nabla F(\bx^k)$, so that the  direction $\bv^k$ is an unbiased estimator of the full gradient of $F$ at $\bx^k$, which is a standard assumption in the related literature. Assumption \ref{ass:b2} on $\EE\left[\norm{\bv^k}^2 \,\vert\, \Fk_k\right]$ is the equivalent of what is called, in the literature \cite{khaled2020better, gorbunov2020unified}, the $ABC$ condition on $\EE\left[\norm{\nabla f_{i_k}(\bx^k)}^2 \,\vert\, \Fk_k\right]$ with  $\sigma_k = \|\nabla F(\bx^k)\|$ and $D$ constant (see also \cite{gorbunov2020unified}).
Assumption \ref{ass:b3} is justified by the fact that it is needed for the theoretical study and it is satisfied by many examples of variance reduction techniques. For additional discussion on these assumptions, especially Assumption \ref{ass:b3}, see \cite{gorbunov2020unified}. 
\end{subsection}
\end{section}
%%%%%%%%%%%%%%%%%%%%%%%%%%%%%%%%%%%%%%%%%%%%%%%%%%%%%%%%%%%%%%%%%%%%%%%%%%%%%%%%%%%%%%%%%%%
%%%%%%%%%%%%%%%%%%%%%%%%%%%%%%%%%%%%%%%%%%%%%%%%%%%%%%%%%%%%%%%%%%%%%%%%%%%%%%%%%%%%%%%%%%%
\begin{section}{Main results}\label{sect:results}
In the rest of the paper, we will always suppose that Assumption \ref{ass:a1} holds.

Before stating the main results of this work, we start with a technical proposition that constitutes the cornerstone of our analysis. The proof can be found in Appendix \ref{app:proofsect3}

\begin{proposition}\label{prop:unified}
Suppose that Assumptions \ref{ass:variance} and \ref{ass:a11} are verified and that the sequence
$(\bx^k)_{k \in \N}$ is generated by Algorithm \ref{algo:unified}. Let $M > 0$. Then, for all $k \in \N$,
\begin{align}
    \EE[\dist(\bx^{k+1}, \argmin F)^2] + \alpha^2M\EE[\sigma_{k+1}^2] &\leq 
     \EE[\dist(\bx^{k}, \argmin F)^2] 
     + \alpha^2 \left[M+B-\rho M\right]\EE[\sigma_k^2] \nonumber \\
     &\qquad -2 \alpha\left[1-\alpha (A+ MC)\right]\EE[F(\bx^{k})- F_*] \nonumber \\
     &\qquad + \alpha^2 \EE[D].  \nonumber
\end{align}
\end{proposition}

We now state two theorems that can be derived from the previous proposition. The first theorem deals with cases where the function $F$ is only convex.

\begin{theorem}\label{theo:unifiedmainonlyconvex}
Suppose that Assumptions \ref{ass:a11} and \ref{ass:variance} hold with $\rho > 0$ and that the sequence $(\bx^k)_{k \in \N}$ is generated by Algorithm \ref{algo:unified}. Let $M>0$ and $\alpha>0$ be such that $M \geq B/\rho$ and $\alpha < 1/(A+MC)$. Then, for all $k \in \N$,
\begin{align}%\label{eq:20221108b}
    \EE[F(\bbar{\bx}^{k})- F_*] &\leq \frac{\dist(\bx^{0}, \argmin F)^{2} + \alpha^2M \EE[\sigma_{0}^2]}{2 \alpha k\left[1-\alpha (A+ MC)\right]} 
     \nonumber,
\end{align}
with $\bbar{\bx}^{k} = \frac{1}{k}\sum_{t=0}^{k-1}\bx^t$.
\end{theorem}
\begin{proof}
Since $B - \rho M \leq 0$ and $\EE[D] \leq 0$, it follows from Proposition \ref{prop:unified} that
\begin{align}%\label{eq:20221108a}
    2 \alpha\left[1-\alpha (A+ MC)\right]\EE[F(\bx^{k})- F_*] &\leq \EE[\dist(\bx^{k}, \argmin F)^2] + \alpha^2M\EE[\sigma_{k}^2] \nonumber \\
    &\qquad - \left(\EE[\dist(\bx^{k+1}, \argmin F)^2] + \alpha^2M\EE[\sigma_{k+1}]\right) \nonumber.
\end{align}
Summing from $0$ up to $k-1$ and dividing both sides by $k$, we obtain
\begin{align}%\label{eq:20221108b}
    2 \alpha\left[1-\alpha (A+ MC)\right]\sum_{t=0}^{k-1}\frac{1}{k}\EE[F(\bx^{t})- F_*] &\leq \frac{1}{k}\left(\dist(\bx^{0}, \argmin F)^{2} + \alpha^2M\EE[\sigma_{0}^2]\right) \nonumber \\
    &\qquad - \frac{1}{k}\left(\EE[\dist(\bx^{k}, \argmin F)]^{2} + \alpha^2M\EE[\sigma_{k}^2]\right) \nonumber \\
    &\leq \frac{1}{k}\left(\dist(\bx^{0}, \argmin F)^{2} + \alpha^2M\EE[\sigma_{0}^2]\right). \nonumber
\end{align}
Finally, by convexity of $F$, we get
\begin{align}
    \EE[F(\bbar{\bx}^{k})- F_*] &\leq \frac{\dist(\bx^{0}, \argmin F)^{2} + \alpha^2M \EE[\sigma_{0}^2]}{2 \alpha k\left[1-\alpha (A+ MC)\right]} 
    \nonumber.
    \qedhere
\end{align}
\end{proof}

The next theorem shows that, when $F$ additionally satisfies the PL property \eqref{PL condition},  the sequence generated by algorithm \ref{algo:unified} exhibits a linear convergence rate both in terms of the distance to a minimizer and also of the values of the objective function. 

\begin{theorem}\label{theo:unifiedmain}
Suppose that Assumptions \ref{ass:funtion} and \ref{ass:variance} are verified with $\rho >0$ and that the sequence $(\bx^k)_{k \in \N}$ is generated by Algorithm \ref{algo:unified}. Let $M$ be such that $M > B/\rho$ and $\alpha > 0$ such that $\alpha < 1 / ( A+MC)$. Set $ q \coloneqq \max\left\{1 - \alpha\mu\left(1-\alpha (A+ MC)\right), 1+\frac{B}{M}-\rho\right\}$. Then $q \in ]0,1[$ and for all $k \in \N$,
\begin{equation*}
%\label{eq:20230311a}
    V^{k+1} \leq q V^k, 
\end{equation*}
with $V^k = \EE[\dist(\bx^{k}, \argmin F)^2] + \alpha^2M\EE[\sigma_{k}^2]\,$ for all $k \in \N$. Moreover, for every $k \in \N$, 
\begin{align*}
    \EE[\dist(\bx^{k}, \argmin F)^2] &\leq  q^{k} \left(\dist(\bx^{0}, \argmin F)^{2} + \alpha^2M\EE[\sigma_{0}^2]\right) ,
    \\[1ex]
    \EE[F(\bx^k)- F_*] &\leq \frac{q^k L}{2} \left(\dist(\bx^{0}, \argmin F)^{2} + \alpha^2M\EE[\sigma_{0}^2]\right).
\end{align*}
\end{theorem}
\begin{proof}
Since $\alpha < \frac{1}{A+MC}$, we obtain thanks to Assumption \ref{ass:a12} and Proposition \ref{prop:unified}
\begin{align}\label{eq:20220927e}
    \EE[\dist(\bx^{k+1}, \argmin F)^2] + \alpha^2M\EE[\sigma_{k+1}^2] &\leq \left[1 - \alpha\mu\left(1-\alpha (A+ MC)\right)\right]\EE[\dist(\bx^{k}, \argmin F)^2] \nonumber \\
    &\qquad + \alpha^2 M\left[1+\frac{B}{M}-\rho\right]\EE[\sigma_k^2].
\end{align}
Since $\alpha < \frac{1}{A+MC}$ and $M > \frac{B}{\rho}$, we obtain from \eqref{eq:20220927e} that
\begin{align}
    &\EE[\dist(\bx^{k+1}, \argmin F)^2] + \alpha^2M\EE[\sigma_{k+1}^2] \nonumber \\
    &\leq \max\left\{1 - \alpha\mu\left(1-\alpha (A+ MC)\right), 1+\frac{B}{M}-\rho\right\} \left(\EE[\dist(\bx^{k}, \argmin F)^2] + \alpha^2M\EE[\sigma_{k}^2]\right) \nonumber.
\end{align}
Since $0<1-\rho \leq 1+B/M -\rho <1$, it is clear that $q \in \left]0,1\right[$. Iterating down on $k$, we obtain
\begin{equation}\label{eq:20230313a}
    \EE[\dist(\bx^{k}, \argmin F)^2] \leq  q^{k} \left(\dist(\bx^{0}, \argmin F)^{2} + \alpha^2M\EE[\sigma_{0}^2]\right).  
\end{equation}
Let $\bxs^* \in \argmin F$. As $F$ is $L$-Lipschitz smooth, from the Descent Lemma \cite[Lemma 1.2.3]{nesterov2003introductory}, we have
\begin{equation}
    (\forall \bxs \in \HHs) \quad F(\bxs) - F_* \leq \frac{L}{2}\|\bxs - \bxs^*\|^2\nonumber.
\end{equation}
In particular,
\begin{equation}\label{eq:20230313b}
    (\forall \bxs \in \HHs) \quad F(\bxs) - F_* \leq \frac{L}{2}\dist(\bxs, \argmin F)^2.
\end{equation}
Using \eqref{eq:20230313b} with $\bx^k$ in \eqref{eq:20230313a}, we get
\begin{equation*}
    \EE[F(\bx^k)- F_*] \leq \frac{q^k L}{2} \left(\dist(\bx^{0}, \argmin F)^{2} + \alpha^2M\EE[\sigma_{0}^2]\right).
    \qedhere
\end{equation*}
\end{proof}
%%%%%%%%%%%%%%%%%%%%%%%%%%%%%%%%%%%%%%%%%%%%%%%%%%%%%%%%%%%%%%%%%%%%%%%%%%%%%%
\begin{remark}\
\begin{enumerate}[label={\rm (\roman*)}]
    \item The convergence rate of order $\grandO{\frac{1}{k}}$ for the general variance reduction scheme \ref{algo:unified}, with constant step size, as stated in Theorem \ref{theo:unifiedmainonlyconvex}, is an improved extension of the one found for the vanilla stochastic proximal gradient method (see e.g. \cite[Proposition $3.8$]{asi2019stochastic}). It is important to mention that the convergence with a constant step size is no longer true when $D$ in Assumption \ref{ass:b2} is positive. As we shall see in Section \ref{sect:appli} several choices for the variance term $\bg_{k}$ in Algorithm \ref{algo:unified} can be beneficial regarding this issue, provided Assumption \ref{ass:b2} with $D\leq 0$. 
    \item The linear rates as stated in Theorem \ref{theo:unifiedmain} have some similarity with the ones found in \cite[Theorem $4.1$]{gorbunov2020unified}, where the authors present a unified study for variance-reduced stochastic (explicit) gradient methods. However we note that the Polyak-{\L}ojasiewicz condition \eqref{PL condition} on $F$ used here is slightly weaker than the quasi-strong convexity used in \cite[Assumption $4.2$]{gorbunov2020unified}.
\end{enumerate}
\end{remark}
\end{section}
%%%%%%%%%%%%%%%%%%%%%%%%%%%%%%%%%%%%%%%%%%%%%%%%%%%%%%%%%%%%%%%%%%%%%%%%%%%%%%%%%%%%%%%%%%%
%%%%%%%%%%%%%%%%%%%%%%%%%%%%%%%%%%%%%%%%%%%%%%%%%%%%%%%%%%%%%%%%%%%%%%%%%%%%%%%%%%%%%%%%%%%
\begin{section}{Derivation of stochastic proximal point type algorithms} \label{sect:appli}
In this section, we provide and analyze several instances of the general scheme \ref{algo:unified}, corresponding to different choices of the variance reduction term $\bg^{k}$. In particular in the next paragraphs we describe four different schemes, namely Stochastic Proximal Point Algorithm (SPPA), Stochastic Variance-reduced Proximal (SVRP) algorithm, Loopless SVRP (L-SVRP) and Stochastic Average Proximal Algorithm (SAPA).

\subsection{Stochastic Proximal Point Algorithm}\label{sect:sppa}
We start by presenting the classic vanilla stochastic proximal method (SPPA), see e.g. 
\cite{bertsekas2011incremental,bianchi2016ergodic,patrascu2017nonasymptotic}. We suppose that
Assumptions \ref{ass:a1} and \ref{ass:a11} hold and that for all $k \in \N$
\begin{equation}
\label{eq:20230328b}
    \sigma^2_k := \sup_{\bxs^* \in \argmin F} \frac 1 n \sum_{i=1}^n \|\nabla f_i(\bxs^*)\|^2<+\infty.
\end{equation}
\begin{algorithm}[\bf{SPPA}] \label{algo:sppa}
Let $(i_k)_{k \in \N}$ be a sequence of i.i.d.~random variables uniformly distributed on $\{1,\dots, n\}$. Let $\alpha_k>0$ for all $k \in \N$ and
set the initial point $\bx^0 \in \HHs$. Define
\begin{equation}\label{eq:20230217e}
\begin{array}{l}
\text{for}\;k=0,1,\ldots\\
\left\lfloor
\begin{array}{l}
\bx^{k+1} = \prox_{\alpha_k f_{i_{k}}}(\bx^k).
\end{array}
\right.
\end{array} 
\end{equation}
\end{algorithm}
Algorithm \ref{algo:sppa} can be directly identified with the general scheme \ref{algo:unified}, by setting $\bg^k = 0$ and $\bv^k = \nabla f_{i_k} (\bx^{k})$. The following lemma provides a bound in expectation on the sequence $\|\bv^{k}\|$ and can be found in the related literature; see, e.g. \cite[Lemma 1]{sebbouh2021almost}.
\begin{lemma}
\label{lem:20230217a}
We suppose that Assumption  \ref{ass:a11} holds and that $(\bx^k)_{k \in \N}$ is a sequence generated by Algorithm \ref{algo:sppa} and $\bv^k = \nabla f_{i_k} (\bx^{k})$. Then, for all $k\in \N$, it holds
\begin{align*}%\label{eq:20230217b}
    &\EE [\| \bv^{k}\|^{2}\,\vert\, \Fk_k] \leq 4 L\big( F(\bx^{k})-F_*\big) + 2 \sigma^2_k, \\
    &\EE\left[\sigma_{k+1}^2\right] = \sigma_{k+1}^2 = \sigma_{k}^2 = \EE\left[\sigma_k^2\right], \nonumber
\end{align*}
where $\Fk_k = \sigma(i_0, \dots, i_{k-1})$ and $\displaystyle \sigma^2_k$ is defined as a constant random variable in \eqref{eq:20230328b}.
\end{lemma}
From Lemma \ref{lem:20230217a}, we immediately notice that Assumptions \ref{ass:variance} are verified with $A = 2L, B =2$,  $C = \rho = 0$ and $D \equiv 0$. In this setting, we are able to recover the following convergence result (see also \cite[Lemma $3.10$ and Proposition $3.8$]{asi2019stochastic}).
\begin{theorem}
    Suppose that Assumption \ref{ass:a11} holds and let $(\gamma_k)_{k \in \N}$ be a positive real-valued sequence such that $\alpha_k \leq \frac{1}{4L}$, for all $k \in \N$. Suppose also that the sequence $(\bx^k)_{k \in \N}$ is generated by Algorithm \ref{algo:sppa} with the sequence $(\alpha_k)_{k \in \N}$. Then, $\forall k \geq 1$,
\begin{equation}
    \EE[F(\bbar{\bx}^{k})- F_*] \leq \frac{\dist(\bx^{0}, \argmin F)^{2}}{\sum_{t=0}^{k-1}\alpha_t} + 2 \sigma^2_1\frac{\sum_{t=0}^{k-1}\alpha_t^2}{\sum_{t=0}^{k-1}\alpha_t}, \nonumber
\end{equation}
where $\displaystyle \bbar{\bx}^{k} = \sum_{t=0}^{k-1}\frac{\alpha_t}{\sum_{t=0}^{k-1}\alpha_t} \bx^t$.
\end{theorem}
\begin{proof}
From Proposition \ref{prop:unified} adapted to the update \eqref{eq:20230217e}, it follows that
\begin{equation}\label{eq:20230217g}
    2 \alpha_k\left[1-2\alpha_k L\right]\EE[F(\bx^{k})- F_*] \leq \EE[\dist(\bx^{k}, \argmin F)]^{2}  - \EE[\dist(\bx^{k+1}, \argmin F)]^{2}  + 2\alpha_k^2 \sigma^2_k. 
\end{equation}
Since $\alpha_k \leq 1/ 4L$, for all $k \in \N$, we have from \eqref{eq:20230217g},
\begin{align}%\label{eq:20230217f}
    \alpha_k\EE[F(\bx^{k})- F_*] &\leq \EE[\dist(\bx^{k}, \argmin F)]^{2}  - \EE[\dist(\bx^{k+1}, \argmin F)]^{2}  + 2\alpha_k^2 \sigma^2_k \nonumber \\
    &= \EE[\dist(\bx^{k}, \argmin F)]^{2}  - \EE[\dist(\bx^{k+1}, \argmin F)]^{2}  + 2\alpha_k^2 \sigma^2_1. \nonumber
\end{align}
Let $k \geq 1$. Summing from $0$ up to $k-1$ and dividing both side by $\sum_{t=0}^{k-1}\alpha_t$, we obtain
\begin{align}%\label{eq:20221108b}
    \sum_{t=0}^{k-1}\frac{\alpha_t}{\sum_{t=0}^{k-1}\alpha_t}\EE[F(\bx^{t})- F_*] &\leq \frac{1}{\sum_{t=0}^{k-1}\alpha_t} \left(\dist(\bx^{0}, \argmin F)^{2}  - \EE[\dist(\bx^{k}, \argmin F)]^{2}\right) \nonumber \\
    &\qquad + 2 \sigma^2_1 \frac{\sum_{t=0}^{k-1}\alpha_t^2}{\sum_{t=0}^{k-1}\alpha_t} \nonumber \\
    &\leq \frac{\dist(\bx^{0}, \argmin F)^{2}}{\sum_{t=0}^{k-1}\alpha_t} + 2 \sigma^2_1\frac{\sum_{t=0}^{k-1}\alpha_t^2}{\sum_{t=0}^{k-1}\alpha_t}. \nonumber
\end{align}
Finally by convexity of $F$ and using Jensen's inequality, we obtain
\begin{equation*}
    \EE[F(\bbar{\bx}^{k})- F_*] \leq \frac{\dist(\bx^{0}, \argmin F)^{2}}{\sum_{t=0}^{k-1}\alpha_t} + 2 \sigma^2_1 \frac{\sum_{t=0}^{k-1}\alpha_t^2}{\sum_{t=0}^{k-1}\alpha_t}.
    \qedhere
\end{equation*}
\end{proof}
%%%%%%%%%%%%%%%%%%%%%%%%%%%%%%%%%%%%%%%%%%%%%%%%%%%%%%%%%%%%%%%%%%%%%%%%%%%%%%
\subsection{Stochastic Variance-Reduced Proximal point algorithm}\label{sect:svrg}
In this paragraph, we present a Stochastic Proximal Point Algorithm coupled with a variance reduction term, in the spirit of the Stochastic Variance Reduction Gradient (SVRG) method introduced in \cite{johnson2013accelerating}. It is coined \emph{Stochastic Variance-Reduced Proximal point algorithm} (SVRP).

The SVRP method involves two levels of iterative procedure: outer iterations and inner iterations. We shall stress out that the framework presented in the previous section covers only the inner iteration procedure and thus the convergence analysis for SVRP demands an additional care. In contrast to the subsequent schemes, Theorems \ref{theo:unifiedmainonlyconvex} and \ref{theo:unifiedmain} do not apply directly to SVRP. In particular, as it can be noted below, in the case of SVRP, the constant $\rho$ appearing in \ref{ass:b3} in Assumption \ref{ass:variance}, is null. Nevertheless, it is worth mentioning that the convergence analysis still uses Proposition \ref{prop:unified}.

\begin{algorithm}[\bf{SVRP}]
\label{algo:srvprox2}
Let $m\in \N$, with $m\geq 1$, and $(\xi_s)_{s \in \N}$, $(i_t)_{t \in \N}$ be two independent sequences of i.i.d.~random variables uniformly distributed on $\{0,1,\dots, m-1\}$ and $\{1,\dots, n\}$ respectively. 
Let $\alpha>0$ and set the initial point $\tilde{\bx}^0 \equiv \tilde{\bxs}^0\in \HHs$. Then
\begin{equation}
\begin{array}{l}
\text{for}\;s=0,1,\ldots\\
\left\lfloor
\begin{array}{l}
\bx^0 = \tilde{\bx}^{s} \\
\text{for}\;k=0,\dots, m-1\\[0.7ex]
\left\lfloor
\begin{array}{l}
\bx^{k+1} = \prox_{\alpha f_{i_{sm+k}}}\left(\bx^k  + \alpha \nabla f_{i_{sm+k}}(\tilde{\bx}^{s}) - \alpha \nabla F(\tilde{\bx}^{s}) \right) 
\end{array}
\right. \\
\tilde{\bx}^{s+1} = \sum_{k=0}^{m-1} \delta_{k,\xi_s}\bx^{k}, \\
% \text{, or} \\
\text{or} \\
{\tilde{\bx}^{s+1} = \frac{1}{m} \sum_{k=0}^{m-1} \bx^{k},}
\end{array}
\right.
\end{array} \nonumber
\end{equation}
where $\delta_{k,h}$ is the Kronecker symbol. { In the case of the first option, one iterate is randomly selected among the inner iterates $\bx_0, \bx_1, \cdots, \bx_{m-1}$, losing possibly a lot of information computed in the inner loop. For the second option, those inner iterates are averaged and most of the information are used.}
\end{algorithm}
Let $s \in \N$. In this case, for all $k \in \{0,1,\cdots, m-1\}$, setting $j_k = i_{sm+k}$, the inner iteration procedure of Algorithm \ref{algo:srvprox2} can be identified with the general scheme \ref{algo:unified}, by setting $\bg^k \coloneqq \nabla f_{j_{k}}(\tilde{\bx}^{s}) - \nabla F(\tilde{\bx}^{s})$. In addition let us define
\begin{equation}\label{defsigmakSVRP}
\sigma_k^2 \coloneqq \frac{1}{n} \sum_{i=1}^{n} \left\|\nabla f_{i}(\tilde{\bx}^{s})-\nabla f_{i}(\tilde{\by}^{s})\right\|^{2} 
\end{equation}
where
 $\tilde{\by}^s \in \argmin F$ is such that $\norm{\tilde{\bx}^s-\tilde{\by}^s} = \dist(\tilde{\bx}^s, \argmin F)$. Moreover, setting
 $\Fk_{s,k} = \sigma(\xi_0, \dots, \xi_{s-1, }i_0, \dots, i_{sm+k-1})$,
we have that $\tilde{\bx}^s, \tilde{\by}^s$, 
and $\bx^k$ are $\Fk_{s,k}$-measurables and $j_k$ is independent of $\Fk_{s,k}$. The following result is proved in Appendix \ref{app:proofsect4}.
\begin{lemma}
\label{lem:20220927a}
Suppose that Assumption \ref{ass:a11} holds true.
Let $s \in \N$ and let $(\bx^k)_{k \in [m]}$ be the (finite) sequence generated by the inner iteration in Algorithm \ref{algo:srvprox2}. Set $\bv^{k}=\nabla f_{j_k}\left(\bx^{k}\right)-\nabla f_{j_k}(\tilde{\bx}^{s})+\nabla F (\tilde{\bx}^{s})$ and $\sigma_{k}$ as defined in \eqref{defsigmakSVRP}. Then, for every $k \in \{0,1,\cdots, m-1\}$, it holds
\begin{equation}\label{eq:20220927c}
    \EE[\| \bv^{k}\|^{2}\,\vert\, \Fk_{s,k}] \leq 4 L\big( F(\bx^{k})-F_*\big) + 2\sigma_k^2 -2\|\nabla F(\tilde{\bx}^{s})\|^2,
\end{equation}
and
\begin{equation*}%\label{eq:20220927d}
       \EE\big[\sigma_{k+1}^2 \,\vert\,\Fk_{s,0}\big] = \EE\big[\sigma_k^2 \,\vert\, \Fk_{s,0}\big].
\end{equation*}
\end{lemma}
As an immediate consequence, of Proposition \ref{prop:unified}, we have the following corollary regarding the inner iteration procedure of Algorithm \ref{algo:srvprox2}.

\begin{corollary}\label{cor:unifiedsvrg}
Suppose that Assumption \ref{ass:a11} holds.
Let $s \in \N$ and let $(\bx^k)_{k \in [m]}$ be the sequence generated by the inner iteration in Algorithm \ref{algo:srvprox2}. Then, for all $k \in \{0,1,\cdots, m-1\}$,
\begin{equation}\label{inner iter for SVRP}
\begin{aligned}
    \EE[\dist(\bx^{k+1}, \argmin F)^2]  &\leq 
     \EE[\dist(\bx^{k}, \argmin F)^2]\\
     &\qquad -2 \alpha\left(1-2\alpha L \right)\EE[F(\bx^{k})- F_*]  \\
     &\qquad + 4 L \alpha^{2} \EE\left[F(\tilde{\bx}^{s})-F_*\right] - 2\alpha^2\EE[\|\nabla F(\tilde{\bx}^{s})\|^2].
     \end{aligned}
\end{equation}
\end{corollary}
\begin{proof}
Since Assumption \ref{ass:a11} is true, by Lemma \ref{lem:20220927a} Assumptions \ref{ass:variance} are satisfied with $\bv^{k}=\nabla f_{j_k}\left(\bx^{k}\right)-\nabla f_{j_k}(\tilde{\bx}^{s})+\nabla F (\tilde{\bx}^{s}), A = 2L, B=2, \rho=C=0,$ and $D = - 2 \|\nabla F(\tilde{\bx}^{s})\|^2$. So Proposition \ref{prop:unified} yields
\begin{equation*}%\label{inner iter for SVRP 2}
\begin{aligned}
    \EE[\dist(\bx^{k+1}, \argmin F)^2 \,\vert\, \Fk_{s,0}]  &\leq 
     \EE[\dist(\bx^{k}, \argmin F)^2 \,\vert\, \Fk_{s,0}]\\
     &\qquad -2 \alpha\left(1-2\alpha L \right)\EE[F(\bx^{k})- F_* \,\vert\, \Fk_{s,0}]  \\
     &\qquad + 4 L \alpha^{2} \EE\left[F(\tilde{\bx}^{s})-F_* \,\vert\, \Fk_{s,0}\right] - 2\alpha^2\EE[\|\nabla F(\tilde{\bx}^{s})\|^2 \,\vert\, \Fk_{s,0}].
     \end{aligned}
\end{equation*}
Thus, taking the total expectation, the statement follows.
\end{proof}
The next theorem shows that, under some additional assumptions on the choice of the step size $\alpha$ and the number of inner iterations $m\in \N$, Algorithm \ref{algo:srvprox2} yields a linear convergence rate in terms of the expectation of the objective function values of the outer iterates ($\tilde{\bx}^{s})_{s \in \N}$.
%%%%%%%%%%%%%%%%%%%%%%%%%%%%%%%%%%%%%%%%%%%%%%%%%%%%%%%%%%%%%%%%%%%%%%%%%%%%%%
\begin{theorem}\label{thm:20221003a}
Suppose that Assumptions \ref{ass:funtion} are satisfied and that the sequence $(\tilde{\bx}^{s})_{s \in \N}$ is generated by Algorithm \ref{algo:srvprox2}
with
\begin{equation}
\label{eq:20230328a}
0<\alpha < \frac{1}{2(2L - \mu)}
\quad\text{and}\quad
m> \frac{1}{\mu \alpha (1 - 2\alpha(2L-\mu))}.
\end{equation}
Then, for all $s\in\N$, it holds
\begin{equation}\label{eq:20230208a}
    \EE\left[F\left(\tilde{\bx}^{s+1}\right)-F_*\right] \leq q^s
    \left(F(\bx^0)- F_*\right),
    \quad\text{with}\quad q \coloneqq \left(\frac{1}{\mu \alpha(1-2 L \alpha) m}+\frac{2\alpha(L - \mu)}{1-2 L \alpha}\right) < 1.
\end{equation}
\end{theorem}
\begin{remark}\ 
\begin{enumerate}[label={\rm (\roman*)}]
\item  Conditions \eqref{eq:20230328a} is used in this form if $\alpha$ is set first and $m$ is chosen after. They are needed to ensure that $q < 1$. Indeed, it is clear that $0<\alpha < \frac{1}{2(2L - \mu)}$ is needed to have $\frac{2\alpha(L - \mu)}{1-2 L \alpha} < 1$. If not, $q$ cannot be less than $1$. Then we need $m> \frac{1}{\mu \alpha (1 - 2\alpha(2L-\mu))}$ once $\alpha$ is fixed.
\item Conditions \eqref{eq:20230328a} can be equivalently stated as follows
\begin{equation*}%\label{eq:20230329a}
    m \geq \frac{8(2L-\mu)}{\mu}
    \quad\text{and}\quad
    \frac{1 - \sqrt{1 - 8(2 \kappa - 1)/m}}{4(2L-\mu)}
     < \alpha \leq
    \frac{1 + \sqrt{1 - 8(2 \kappa - 1)/m}}{4(2L-\mu)},
    \quad \kappa = \frac L \mu.
\end{equation*}
The above formulas can be useful if one prefers to set the parameter $m$ first and set the step size $\alpha$ afterwards.
\item The convergence rate in Theorem~\ref{thm:20221003a} establishes the improvement from the outer step $s$ to $s+1$.  Of course, it depends on the number of inner iterations $m$. As expected, and as we can see from Equation \ref{eq:20230208a},  increasing $m$ improve the bound on the rate, and since $m$ is not bounded from above,  the best choice would be to let $m$ go to $+\infty$. In practice, there is no best choice of $m$, but empirically a balance between the number of inner and outer iterations should be found. Consequently, there is also no optimal choice  for $\alpha$ either.
{
\item It is worth mentioning that the linear convergence factor $q$ in \eqref{eq:20230208a} is better (smaller) than the one provided in \cite[Theorem $1$]{johnson2013accelerating} for the SVRG method for strongly convex functions. There, it is $\frac{1}{\mu \alpha(1-2 L \alpha) m}+\frac{2\alpha L }{1-2 L \alpha}$. The linear convergence factor in \eqref{eq:20230208a} is also better than the one in \cite[Proposition $3.1$]{zhang2022linear}, and also \cite[Theorem $1$]{gong2014linear}, dealing with a proximal version of SVRG for functions satisfying the PL condition \eqref{PL condition}. In both papers, it is $\frac{1}{2\mu \alpha(1-4 L \alpha) m}+\frac{4\alpha L(m+1)}{(1-4 L \alpha)m}$. However, we note that this improvement can also be obtained for the aforementioned SVRG methods using a similar analysis.
\item  We stress that following the lines of the proof of Theorem \ref{thm:20221003a}, the same linear convergence rate found in \eqref{eq:20230208a}, holds true also for the averaged iterate $\tilde{\bx}^{s+1} = \frac{1}{m} \sum_{k = 0}^{m-1}  \bx^k$. But, the analysis does not work for the last iterate of the inner loop.
}
\end{enumerate}
\end{remark}
{
\begin{remark}
In \cite{milzarek2022semismooth} a variant of Algorithm \ref{algo:srvprox2} (SVRP) in this paper, called SNSPP,  has been also proposed and analyzed. The SNSPP algorithm includes a subroutine to compute the proximity operator. This can be useful in practice when a closed form solution of the proximity operator is not available. Contingent on some additional conditions on the conjugate of $f_i$ and assuming semismoothness of the proximity mapping,  \cite{milzarek2022semismooth} provides a linear convergence rate for SNSPP for $F$ Lipschitz smooth and strongly convex. An ergodic sublinear convergence rate is also proved for weakly convex functions.
\end{remark}
}
\begin{proof}[Proof of Theorem \ref{thm:20221003a}]
We consider a fixed stage $s \in \N$
and $\tilde{\bx}^{s+1}$ is defined as in Algorithm \ref{algo:srvprox2}. By summing inequality \eqref{inner iter for SVRP} in Corollary \ref{cor:unifiedsvrg} over $k=0, \ldots, m-1$ and taking the total expectation, we obtain
\begin{equation}\label{for averaging}
\begin{aligned}
\EE[\dist(\bx^{m}, \argmin F)]^{2} &+ 2 \alpha(1-2 L \alpha) m \EE[F\left(\tilde{\bx}^{s+1}\right)-F_*] \\
&\leq \EE[\dist(\bx^{0}, \argmin F)]^{2}+4 L m \alpha^{2} \EE\left[F(\tilde{\bx}^{s})-F_*\right] - 2\alpha^2m\EE[\|\nabla F(\tilde{\bx}^{s})\|^2]\\
&= \EE[\dist(\tilde{\bx}^{s}, \argmin F)]^{2}+4 L m \alpha^{2} \EE\left[F(\tilde{\bx}^{s})-F_*\right] -2\alpha^2m\EE[\|\nabla F(\tilde{\bx}^{s})\|^2]\\
&\leq \frac{2}{\mu} \EE\left[F(\tilde{\bx}^{s})-F_*\right]+4 L m \alpha^{2} \EE\left[F(\tilde{\bx}^{s})-F_*\right]-2\alpha^2m\EE[\|\nabla F(\tilde{\bx}^{s})\|^2] \\
&= 2\left(\mu^{-1}+2 L m \alpha^{2}\right) \EE\left[F(\tilde{\bx}^{s})-F_*\right]-2\alpha^2m\EE[\|\nabla F(\tilde{\bx}^{s})\|^2] \\
 &\leq 2\left(\mu^{-1}+2 L m \alpha^{2} - 2\mu m \alpha ^2\right) \EE\left[F(\tilde{\bx}^{s})-F_*\right].
\end{aligned}
\end{equation}
In the first inequality, we used the fact that
\begin{equation*}
\sum_{k = 0}^{m-1} F(\bx^k) = m \sum_{k = 0}^{m-1} \frac{1}{m} F(\bx^k) = m \sum_{\xi = 0}^{m-1} \frac{1}{m} F\left({\textstyle\sum_{k=0}^{m-1} \delta_{k,\xi}\bx^k}\right) = m \EE[F(\tilde{\bx}^{s+1}) \,\vert\, \Fk_{s,m-1}].
\end{equation*}
Notice that relation \eqref{for averaging} is still valid by choosing $\tilde{\bx}^{s+1} = \sum_{k = 0}^{m-1} \frac{1}{m} \bx^k$ , in Algorithm \ref{algo:srvprox2}, and using Jensen inequality to lower bound $\sum_{k = 0}^{m-1} F(\bx^k)$ by $m F(\tilde{\bx}^{s+1})$. 

The second and the last inequalities use, respectively, the quadratic growth \eqref{eq:mu_strongly_quasi_convex} and the PL condition \eqref{PL condition}. We thus obtain
\begin{equation*}
\EE\left[F\left(\tilde{\bx}^{s+1}\right)-F_*\right] \leq\left(\frac{1}{\mu \alpha(1-2 L \alpha) m}+\frac{2\alpha(L - \mu)}{1-2 L \alpha}\right) \EE\left[F\left(\tilde{\bx}^{s}\right)-F_*\right]. 
\qedhere
\end{equation*}
\end{proof}
%%%%%%%%%%%%%%%%%%%%%%%%%%%%%%%%%%%%%%%%%%%%%%%%%%%%%%%%%%%%%%%%%%%%%%%%%%%%%%%%%%%%%%%%%%%
%%%%%%%%%%%%%%%%%%%%%%%%%%%%%%%%%%%%%%%%%%%%%%%%%%%%%%%%%%%%%%%%%%%%%%%%%%%%%%
\subsection{Loopless SVRP}
In this paragraph, we propose a single-loop variant of the SVRP algorithm presented previously, by removing the burden of choosing the number of inner iterations. This idea is inspired by the Loopless Stochastic Variance-Reduced Gradient (L-SVRG) method, as proposed in \cite{kovalev2020don,hofmann2015variance} (see also \cite{gorbunov2020unified}) and here, we present the stochastic proximal method variant that we call L-SVRP.
%%%%%%%%%%%%%%%%%%%%%%%%%%%%%%%%%%%%%%%%%%%%%%%%%%%%%%%%%%%%%%%%%%%%%%%%%%%%%%
\begin{algorithm}[\bf{L-SVRP}]
\label{algo:srvp-l}
Let $(i_k)_{k \in \N}$ be a sequence of i.i.d.~random variables uniformly distributed on $\{1,\dots, n\}$ and let $(\varepsilon^k)_{k \in \N}$ be a sequence of i.i.d~Bernoulli random variables such that $\mathsf{P}(\varepsilon^k=1)=p \in \left]0,1\right]$. Let $\alpha>0$ and
set the initial points $\bx^0 = \bu^0 \in \HHs$. Then
\begin{equation}
\begin{array}{l}
\text{for}\;k=0,1,\ldots\\
\left\lfloor
\begin{array}{l}
\bx^{k+1} = \prox_{\alpha f_{i_{k}}} \left(\bx^k + \alpha \nabla f_{i_{k}}(\bu^{k}) - \alpha \nabla F(\bu^{k}) \right)  \\
  \bu^{k+1}=(1-\varepsilon^k)\bu^k + \varepsilon^k \bx^k
\end{array}
\right.
\end{array} \nonumber
\end{equation}
\end{algorithm}
%%%%%%%%%%%%%%%%%%%%%%%%%%%%%%%%%%%%%%%%%%%%%%%%%%%%%%%%%%%%%%%%%%%%%%%%%%%%%%
Here we note that Algorithm \ref{algo:srvp-l} can be identified with the general scheme \ref{algo:unified}, by setting $\bg^k \coloneqq \nabla f_{i_{k}}(\bu^{k}) - \nabla F(\bu^{k})$. In addition, we define
\begin{equation}\label{defsigmakloopless}
\sigma_k^2 \coloneqq \frac{1}{n} \sum_{i=1}^{n} \big\|\nabla f_{i}(\bu^{k})-\nabla f_{i}(\by^{k})\big\|^{2}
\end{equation}
with $\by^{k} \in \argmin F$ such that $\norm{\bu^{k}-\by^{k}} = \dist(\bu^{k}, \argmin F)$.
Moreover, setting
 $\Fk_k = \sigma(i_0, \dots, i_{k-1}, \varepsilon^0, \dots, \varepsilon^{k-1})$,
we have that $\bx^k$, $\bu^k$ and $\by^k$ are $\Fk_k$-measurable, $i_k$ and $\varepsilon^k$ are independent of $\Fk_k$.
%%%%%%%%%%%%%%%%%%%%%%%%%%%%%%%%%%%%%%%%%%%%%%%%%%%%%%%%%%%%%%%%%%%%%%%%%%%%%%
\begin{lemma}
\label{lem:20221103a}
Suppose that Assumption \ref{ass:a11} is satisfied. Let $(\bx^k)_{k \in \N}$ be the sequence generated by Algorithm \ref{algo:srvp-l}, with $\bv^{k}=\nabla f_{i_k}(\bx^{k})-\nabla f_{i_k}(\bu^{k})
+\nabla F (\bu^{k})$ and $\sigma_{k}$ as defined in \eqref{defsigmakloopless}. Then for all $k\in\N$, it holds
\begin{equation}\label{eq:20221103a}
    \EE[\| \bv^{k}\|^{2}\,\vert\, \Fk_k] \leq 4 L\big( F(\bx^{k})-F_*\big) + 2\sigma_k^2,
\end{equation}
and
\begin{equation}\label{eq:20221103b}
        \EE\big[\sigma_{k+1}^2\big] \leq (1-p)\EE\big[\sigma_k^2\big] + 2 p L \EE\big[ F(\bx^{k})-F_*\big].
\end{equation}
\end{lemma}
%%%%%%%%%%%%%%%%%%%%%%%%%%%%%%%%%%%%%%%%%%%%%%%%%%%%%%%%%%%%%%%%%%%%%%%%%%%%%%
Lemma \ref{lem:20221103a} whose proof can be found in Appendix \ref{app:proofsect4} ensures that Assumptions \ref{ass:variance} hold true with constants $A = 2L, B=2, C = pL$, $\rho=p$ and $D \equiv 0$. Then the following corollaries can be obtained by applying respectively Theorem \ref{theo:unifiedmainonlyconvex} and \ref{theo:unifiedmain} on Algorithm \ref{algo:srvp-l}.
\begin{corollary}\label{cor:lsvrponlyconvex}
Suppose that Assumption \ref{ass:a11} holds. Suppose also that the sequence $(\bx^k)_{k \in \N}$ is generated by Algorithm \ref{algo:srvp-l}. Let $M$ such that $M \geq \frac{2}{p}$ and $\alpha > 0$ such that $\alpha < \frac{1}{L(2 + p M)}$. Then, for all $k \in \N$,
\begin{align}%\label{eq:20221108b}
    \EE[F(\bbar{\bx}^{k})- F_*] &\leq \frac{\dist(\bx^{0}, \argmin F)^{2} + \alpha^2M\EE[\sigma_{0}^2]}{2 \alpha k\left[1-\alpha L(2+ pM)\right]} \nonumber,
\end{align}
with $\bbar{\bx}^{k} = \frac{1}{k}\sum_{t=0}^{k-1}\bx^t$.
\end{corollary}
\begin{corollary}\label{cor:lsvrp}
Suppose that Assumptions \ref{ass:funtion} are verified. Suppose now that the sequence $(\bx^k)_{k \in \N}$ is generated by Algorithm \ref{algo:srvp-l}. Let $M$ such that $M > 2/p$ and $\alpha > 0$ such that $\alpha < \frac{1}{L(2 + p M)}$. Set $ q \coloneqq \max\left\{1 - \alpha\mu\left(1-\alpha L(2 + p M)\right), 1+\frac{2}{M}-p\right\}$. Then $q \in ]0,1[$ and for all $k \in \N$,
\begin{align*}
    \EE[\dist(\bx^{k}, \argmin F)]^{2} &\leq  q^{k} \left(\dist(\bx^{0}, \argmin F)^{2} + \alpha^2M\EE[\sigma_{0}^2]\right)\nonumber,
    \\[1ex]
    \EE[F(\bx^k)- F_*] &\leq \frac{q^k L}{2} \left(\dist(\bx^{0}, \argmin F)^{2} + \alpha^2M\EE[\sigma_{0}^2]\right).
\end{align*}
\end{corollary}

{
\begin{remark}
The proximal version of L-SVRG \cite{kovalev2020don} has been concurrently proposed in~\cite{khaled2022faster}. Linear convergence holds in the smooth strongly convex setting.  In \cite{khaled2022faster}, an approximation of the proximity operator at each iteration is used. Also, Lipschitz continuity of the gradient is replaced by the weaker ``second-order similarity'', namely
\begin{equation*}
    \frac{1}{n} \sum_{i=1}^n\left\|\nabla f_i(\bxx)-\nabla f(\bxx)-\left[\nabla f_i(\byy)-\nabla f(\byy)\right]\right\|^2 \leq \delta^2\|\bxx-\byy\|^2.
\end{equation*}
% But one can argue that, since $F$ is strongly convex, it will be more beneficial to do coordinate descent on the dual problem. Aside from that result, they have, for constrained version, a linear convergence rate under the same assumptions.
In \cite{khaled2022faster} the  convex case is not analyzed.
\end{remark}
}
%%%%%%%%%%%%%%%%%%%%%%%%%%%%%%%%%%%%%%%%%%%%%%%%%%%%%%%%%%%%%%%%%%%%%%%%%%%%%%%%%%%%%%%%%%%
\subsection{Stochastic Average Proximal Algorithm}
In this paragraph, we propose a new stochastic proximal point method in analogy to SAGA \cite{defazio2014saga}, called \emph{Stochastic Aggregated Proximal Algorithm} (SAPA).
\begin{algorithm}[\bf{SAPA}] \label{algo:sagaprox}
Let $(i_k)_{k \in \N}$ be a sequence of i.i.d.~random variables uniformly distributed on $\{1,\dots, n\}$. Let $\alpha>0$ and set, for every $i \in [n]$, $\bphi_i^0=\bx^0 \in \HHs$. Then
\begin{equation}
\begin{array}{l}
\text{for}\;k=0,1,\ldots\\
\left\lfloor
\begin{array}{l}
\bx^{k+1} = \prox_{\alpha f_{i_{k}}}\left(\bx^k + \alpha \nabla f_{i_k} (\bphi_{i_k}^k) - \frac{\alpha}{n} \sum_{i=1}^n \nabla f_i (\bphi_i^k) \right), \\
\forall\, i \in [n]\colon\ \bphi^{k+1}_i = \bphi^k_i + \delta_{i,i_k} (\bx^{k} - \bphi^k_i),
\end{array}
\right.
\end{array} \nonumber
\end{equation}
where $\delta_{i,j}$ is the Kronecker symbol.
\end{algorithm}
{
\begin{remark}
Algorithm \ref{algo:sagaprox} is similar but different to the  Point-SAGA one proposed in \cite{defazio2016simple}. The update of the stored gradients for Point-SAGA  is
\begin{equation*}
    \forall\, i \in [n]\colon\ \bphi^{k+1}_i = \bphi^k_i + \delta_{i,i_k} (\bx^{k+1} - \bphi^k_i),
\end{equation*}
whereas in SAGA \cite{defazio2014saga} and its proximal version that we proposed in Algorithm \ref{algo:sagaprox}, the update is
\begin{equation*}
    \forall\, i \in [n]\colon\ \bphi^{k+1}_i = \bphi^k_i + \delta_{i,i_k} (\bx^{k} - \bphi^k_i).
\end{equation*}
\end{remark}
}
As for the previous cases, SAPA can be identified with Algorithm \ref{algo:unified}, by setting $\bg^k \coloneqq \nabla f_{i_{k}}(\bphi_{i_k}^k) - \frac{1}{n} \sum_{i=1}^n \nabla f_i (\bphi_i^k)$ for all $k\in \N$. In addition, let
\begin{equation}\label{sigmakSAPA}
\sigma_k^2 \coloneqq \frac{1}{n} \sum_{i=1}^{n} \big\|\nabla f_{i}(\bphi_i^k)-\nabla f_{i}(\bx^{*})\big\|^{2}
\end{equation} 
with $\bx^{*} \in \argmin F$ such that $\norm{\bx^0-\bx^{*}} = \dist(\bx^0, \argmin F)$. Setting
 $\Fk_k = \sigma(i_0, \dots, i_{k-1})$,
we have that $\bx^k$ and $\bphi_i^k$ are $\Fk_k$-measurables and $i_k$ is independent of $\Fk_k$.
%%%%%%%%%%%%%%%%%%%%%%%%%%%%%%%%%%%%%%%%%%%%%%%%%%%%%%%%%%%%%%%%%%%%%%%%%%%%%%%%%%%%%%%%%
\begin{lemma}
\label{lem:20221004a}
Suppose that Assumption \ref{ass:a11} holds. Let $(\bx^k)_{k \in \N}$ be the sequence generated by Algorithm \ref{algo:sagaprox}, with $\bv^{k}=\nabla f_{i_k}\left(\bx^{k}\right)-\nabla f_{i_k}(\bphi_{i_k}^k)+\frac{1}{n} \sum_{i=1}^n \nabla f_i (\bphi_i^k)$ and $\sigma_{k}$ as defined in \eqref{sigmakSAPA}. Then, for all $k\in \N$, it holds

\begin{equation*}%\label{eq:20221004b}
    \EE\big[\| \bv^{k}\|^{2}\,\vert\, \Fk_k\big] \leq 4 L( F(\bx^{k})-F_*) + 2\sigma_k^2, 
\end{equation*}
and
\begin{equation}\label{eq:20221004c}
        \EE\left[\sigma_{k+1}^2\right] \leq \left(1 - \frac{1}{n}\right) \EE\left[\sigma_k^2 \right] + \frac{2L}{n}\EE\big[F(\bx^k) - F_*\big].
\end{equation}
\end{lemma}
%%%%%%%%%%%%%%%%%%%%%%%%%%%%%%%%%%%%%%%%%%%%%%%%%%%%%%%%%%%%%%%%%%%%%%%%%%%%%%
From the above lemma, we know that Assumptions \ref{ass:variance} are verified with $A = 2L, B=2, C = \frac{L}{n}$, $\rho=\frac{1}{n}$ and $D \equiv 0$. These allow us to state the next corollaries obtained by applying respectively Theorem \ref{theo:unifiedmainonlyconvex} and \ref{theo:unifiedmain} on Algorithm \ref{algo:sagaprox}.
\begin{corollary}
Suppose that Assumptions \ref{ass:a11} are verified. Suppose also that the sequence $(\bx^k)_{k \in \N}$ is generated by Algorithm \ref{algo:sagaprox}. Let $M$ such that $M \geq 2n$ and $\alpha > 0$ such that $\alpha < \frac{1}{L(2+M/n)}$. Then, for all $k \in \N$,
\begin{align}%\label{eq:20221108b}
    \EE[F(\bbar{\bx}^{k})- F_*] &\leq \frac{\dist(\bx^{0}, \argmin F)^{2} + \alpha^2M\EE[\sigma_{0}^2]}{2 \alpha k\left[1-\alpha L(2+ M/n)\right]} \nonumber,
\end{align}
with $\bbar{\bx}^{k} = \frac{1}{k}\sum_{t=0}^{k-1}\bx^t$.
\end{corollary}
%%%%%%%%%%%%%%%%%%%%%%%%%%%%%%%%%%%%%%%%%%%%%%%%%%%%%%%%%%%%%%%%%%%%%%%%%%%%%%
\begin{corollary}\label{cor:saga}
Suppose that Assumptions \ref{ass:funtion} hold. Suppose now that the sequence $(\bx^k)_{k \in \N}$ is generated by Algorithm \ref{algo:sagaprox}. Let $M$ such that $M > 2n$ and $\alpha > 0$ such that $\alpha < \frac{1}{L(2+M/n)}$. Set $q \coloneqq \max\left\{1 - \alpha\mu\left(1-\alpha L(2+M/n)\right), 1+\frac{2}{M}-\frac{1}{n}\right\}$. Then $q \in ]0,1[$ and for all $k \in \N$,
\begin{align*}
    \EE[\dist(\bx^{k}, \argmin F)]^{2} &\leq  q^{k} \left(\dist(\bx^{0}, \argmin F)^{2} + \alpha^2M\EE[\sigma_{0}^2]\right)\nonumber,
        \\[1ex]
    \EE[F(\bx^k)- F_*] &\leq \frac{q^k L}{2} \left(\dist(\bx^{0}, \argmin F)^{2} + \alpha^2M\EE[\sigma_{0}^2]\right).
\end{align*}
\end{corollary}

\begin{remark}
Now, we compare our results with those in \cite{defazio2016simple}  for Point-SAGA. In the smooth case, Point-SAGA converges linearly when  $f_i$  is differentiable with $L-$ Lipschitz gradient and $\mu-$strongly convex for every $i \in [n]$, whereas we require convexity of $f_i$ and only the PL condition to be satisfied by $F$. The work \cite{defazio2016simple} does not provide any rate for convex functions but instead does provide an ergodic sublinear convergence rate for nonsmooth and strongly convex functions.  While the rate is ergodic and sublinear for strongly convex functions, the algorithm converges with a constant step size.
\end{remark}

\end{section}
%%%%%%%%%%%%%%%%%%%%%%%%%%%%%%%%%%%%%%%%%%%%%%%%%%%%%%%%%%%%%%%%%%%%%%%%%%%%%%%%%%%%%%%%%%%
%%%%%%%%%%%%%%%%%%%%%%%%%%%%%%%%%%%%%%%%%%%%%%%%%%%%%%%%%%%%%%%%%%%%%%%%%%%%%%%%%%%%%%%%%%%
{
\begin{section}{Experiments}\label{sect:experiments}
In this section, we perform some experiments on synthetic data to compare the schemes presented and analyzed in Section \ref{sect:appli}. We compare the variance-reduced algorithms SAPA (Algorithm \ref{algo:sagaprox}) and SVRP (Algorithm \ref{algo:srvprox2}) with their vanilla counterpart SPPA (Algorithm \ref{algo:sppa}) and their explicit gradient counterparts: SAGA \cite{defazio2014saga}  and SVRG \cite{johnson2013accelerating}. The plots presented in the section are averaged of 10 or 5 runs, depending on the computational demand of the problem. The deviation from the average is also plotted. All the codes
are available on GitHub\footnote{https://github.com/cheiktraore/Variance-reduction-for-SPPA}.

\subsection{Comparing SAPA and SVRP to SPPA}
The cost of each iteration of SVRP is different from that of SPPA. More precisely, SVRP consists of two nested iterations, where each outer iteration requires a full explicit gradient and $m$ stochastic gradient computations.  We will consider the minimization of the sum of $n$ functions; therefore, we run SPPA for $N$ iterations, with $N = S(m + n + 1)$, where $S$ is the maximum number of outer iterations and $m$ is that of inner iterations as defined in Algorithm \ref{algo:srvprox2}. Let $s$ be the outer iterations counter. As for $m$, it is set at $2n$ like in \cite{johnson2013accelerating}. Then the step size $\alpha$ in SVRP is fixed at $1/(5L)$. For SAPA, we run it for $N - n$ iterations because there is a full gradient computation at the beginning of the algorithm. The SAPA step size is set to $1 / (5L)$. Finally, the SPPA step size is is chosen to be $\alpha_k = 1 / (k^{0.55})$.

For all three algorithms, we normalize the abscissa so  to present the convergence with respect to the outer iterates $(\tilde{\bx}^s)_{s \in \N}$ as in Theorem \ref{thm:20221003a}.  

The algorithms are run for $n \in \{1000, 5000, 10000\}$, $d$ fixed at $500$ and $cond(A^{\top\!}A)$, the condition number of $A^{\top\!}A$, at $100$.

\subsubsection{Logistic regression}\label{sect:logistic}
First, we consider experiments on the logistic loss:
\begin{equation}\label{prob:logistic}
    F(\bxs) = \frac{1}{n} \sum_{i=1}^n \log(1+\exp\{-b_i\langle \bas_i, \bxs \rangle\}),
\end{equation}
where $\bas_i$ is the $i^{th}$ row of a matrix $A \in \R^{n \times d}$ and $b_i\in \{-1, 1\}$ for all $i \in [n]$. If we set $f_i(\bxs) = \log(1+\exp\{-b_i\langle \bas_i, \bxs\rangle\})$ for all $i \in [n]$, then $F(\bxs) = \sum_{i=1}^n \frac{1}{n} f_i(\bxs)$.
The matrix $A$ is generated randomly. We first generate a matrix $M$ according to the standard normal distribution. A singular value decomposition gives $M = UDV^{\top\!}$. We set the smallest singular value to zero and rescale the rest of the vector of singular values so that the biggest singular value is equal to a given condition number and the second smallest to one. We obtain a new diagonal matrix $D^{\prime}$. Then $A$ is given by $UD^{\prime}V^{\top\!}$. In this problem, we have $L = 0.25 \times \max_i \|\bas_i\|_2^2$.  We compute the proximity operator of the logistic function $f_i$ according to the formula and the subroutine code available in \cite{chierchia2020proximity} and considering the rule of calculus of the proximity operator of a function composed with a linear map \cite[Corollary 24.15]{bauschke2017correction}.

Even though we don't have any theoretical result for SVRP in the convex case, we perform some experiments in this case as well. As it can readily be seen in Figure \ref{fig:logcompsvrpstochprox}, both SAPA and SVRP are better than SPPA. 
%%%%%%%%%%%%%%%%%%%%%%%%%%%%%%%%%%%%%%%%%%%%%%%%%%%%%%%%%%%%%%%%%%%%%%%%%%%%%%
\subsubsection{Ordinary least squares (OLS)}\label{sect:ols}
To analyze the practical behavior of the proposed methods when the PL condition holds, we  test all the algorithms on an ordinary least squares problem: 
\begin{equation}
      \label{prob:ols}
    \minimize{ \bxs \in \R^d}{F(\bxs)} = \frac{1}{2n} \|A\bxs - \bb \|^2=\frac{1}{n}\sum_{i = 1}^{n}\frac{1}{2}\left(\langle\bas_i, \bxs \rangle - b_i \right)^2,
\end{equation}
where $\bas_i$ is the $i^{th}$ row of the matrix $A\in\mathbb{R}^{n\times d}$ and $b_i\in\mathbb{R}$ for all $i\in [n]$. In this setting, we have $F(\bxs) = \sum_{i=1}^n \frac{1}{n} f_i(\bxs)$ with 
$f_i(\bxs) = \frac{1}{2}\left(\langle\bas_i, \bxs \rangle - b_i \right)^2$ for all $i \in [n]$.

Here, the matrix $A$ was generated as in  Section \ref{sect:logistic}. The proximity operator is computed with the following closed form solution:
\[
\prox_{\alpha f_i}(\bx) = \bx + \alpha  \frac{b_i - \langle \bas_i, \bx\rangle}{\alpha\|\bas_i\|^2 + 1} \bas_i.
\]
Like in the logistic case, SVRP and SAPA exhibit faster convergence compared to SPPA. See Figure \ref{fig:olscompsvrpstochprox}.
%%%%%%%%%%%%%%%%%%%%%%%%%%%%%%%%%%%%%%%%%%%%%%%%%%%%%%%%%%%%%%%%%%%%%%%%%%%%%%
\begin{figure}[ht]
     \centering
     \begin{subfigure}[b]{0.32\textwidth}
         \centering
         \includegraphics[width=\textwidth]
         {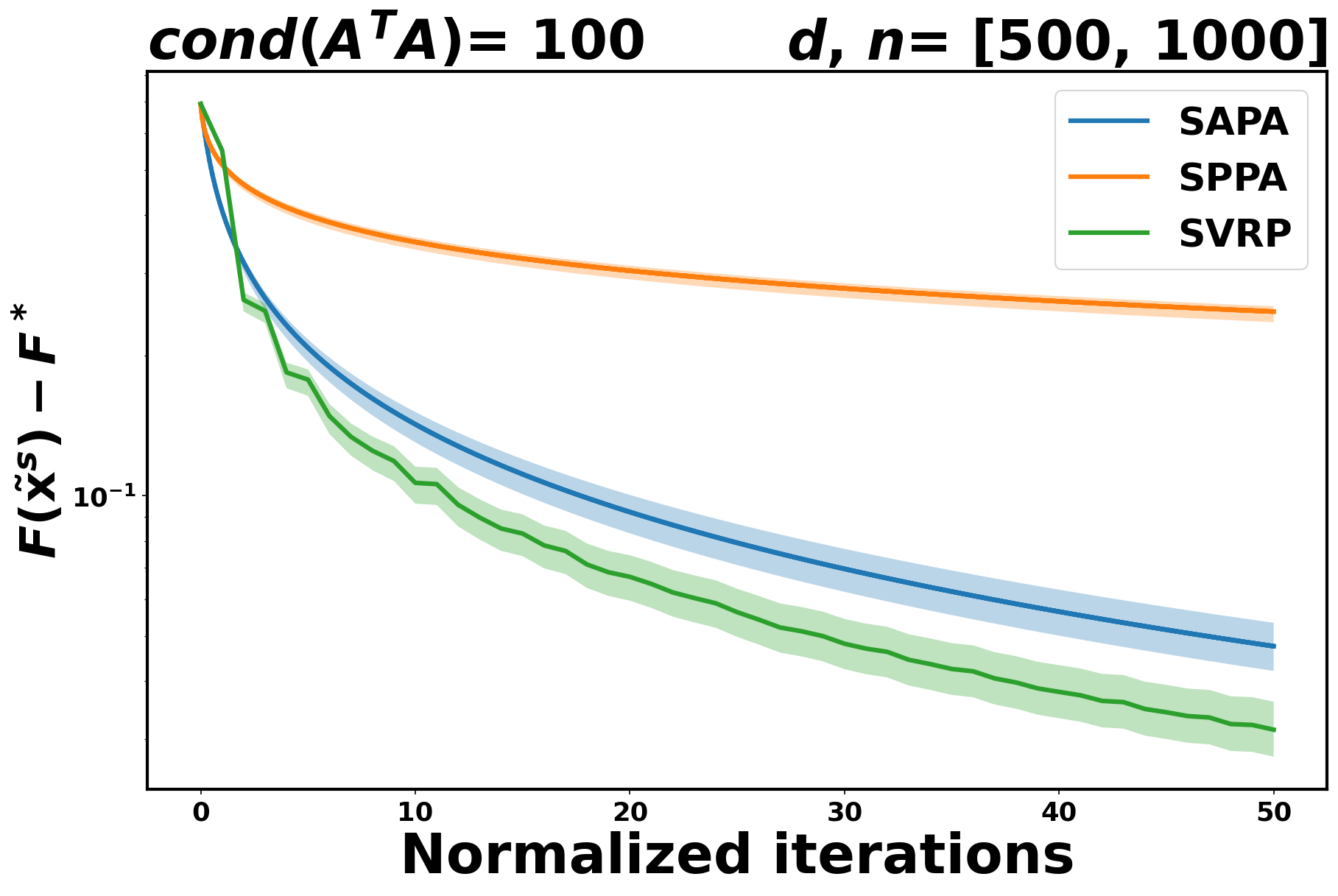}
     \end{subfigure}
     \begin{subfigure}[b]{0.32\textwidth}
         \centering
         \includegraphics[width=\textwidth]
         {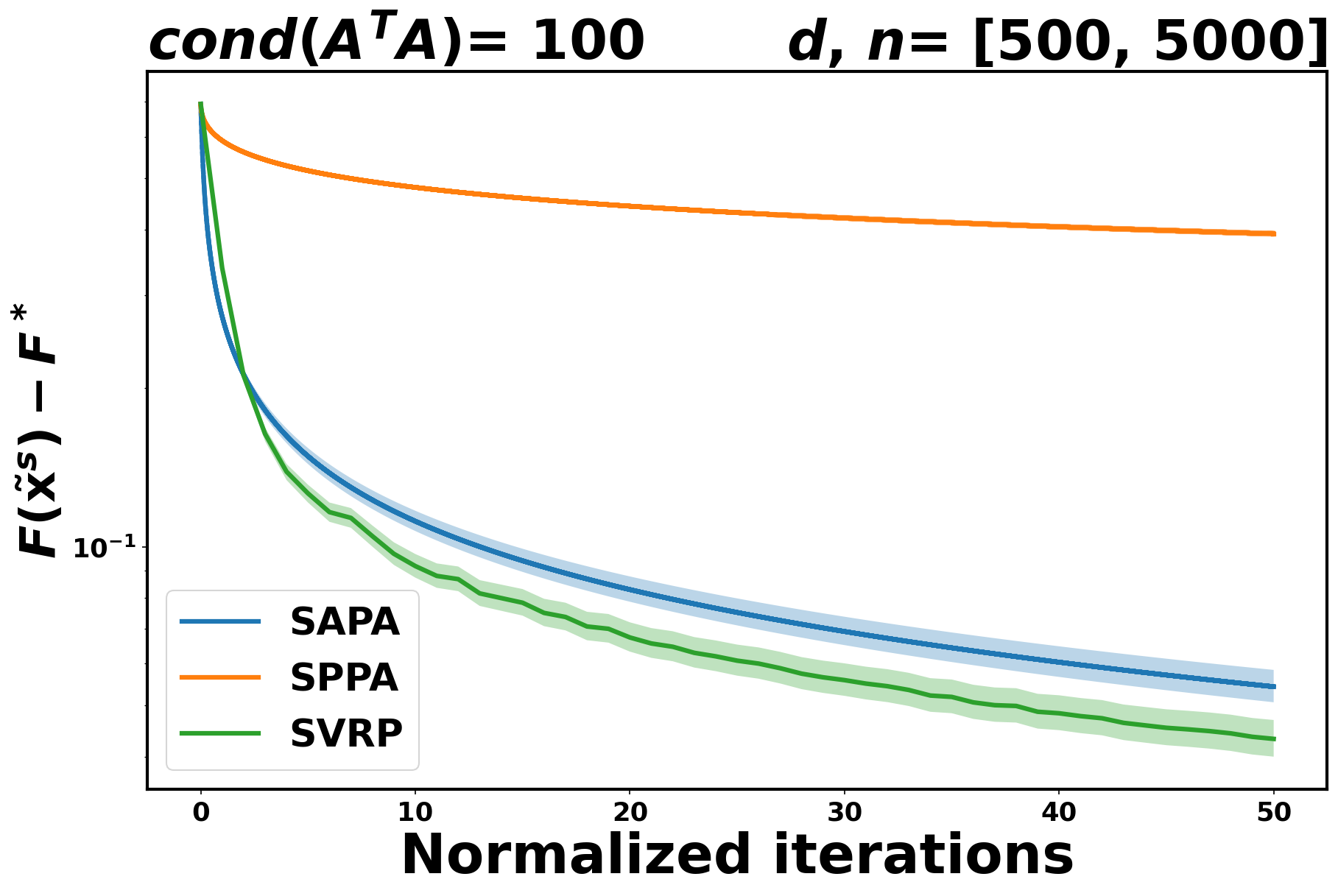}
     \end{subfigure}
    \begin{subfigure}[b]{0.32\textwidth}
         \centering
         \includegraphics[width=\textwidth]
         {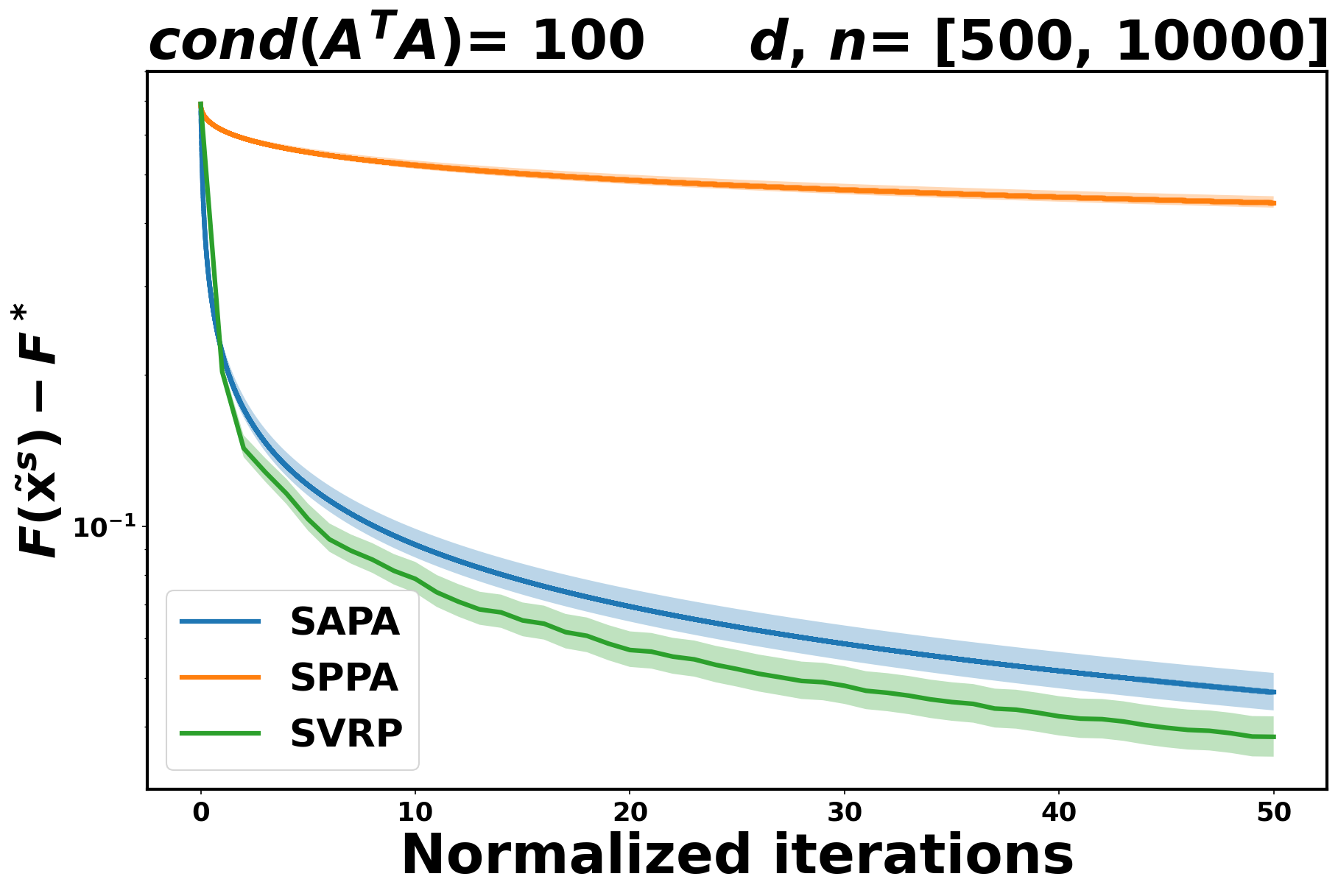}
     \end{subfigure}
    \caption{Evolution of $F(\tilde{\bx}^s)-F_*$, with respect to the normalized iterations counter, for different values of number of functions $n$.We compare the performance of SAPA (blue) and SVRP (green) to SPPA (orange) in the logistic regression case.
    }
    \label{fig:logcompsvrpstochprox}
\end{figure}

\begin{figure}[ht]
     \centering
     \begin{subfigure}[b]{0.32\textwidth}
         \centering
         \includegraphics[width=\textwidth]
         {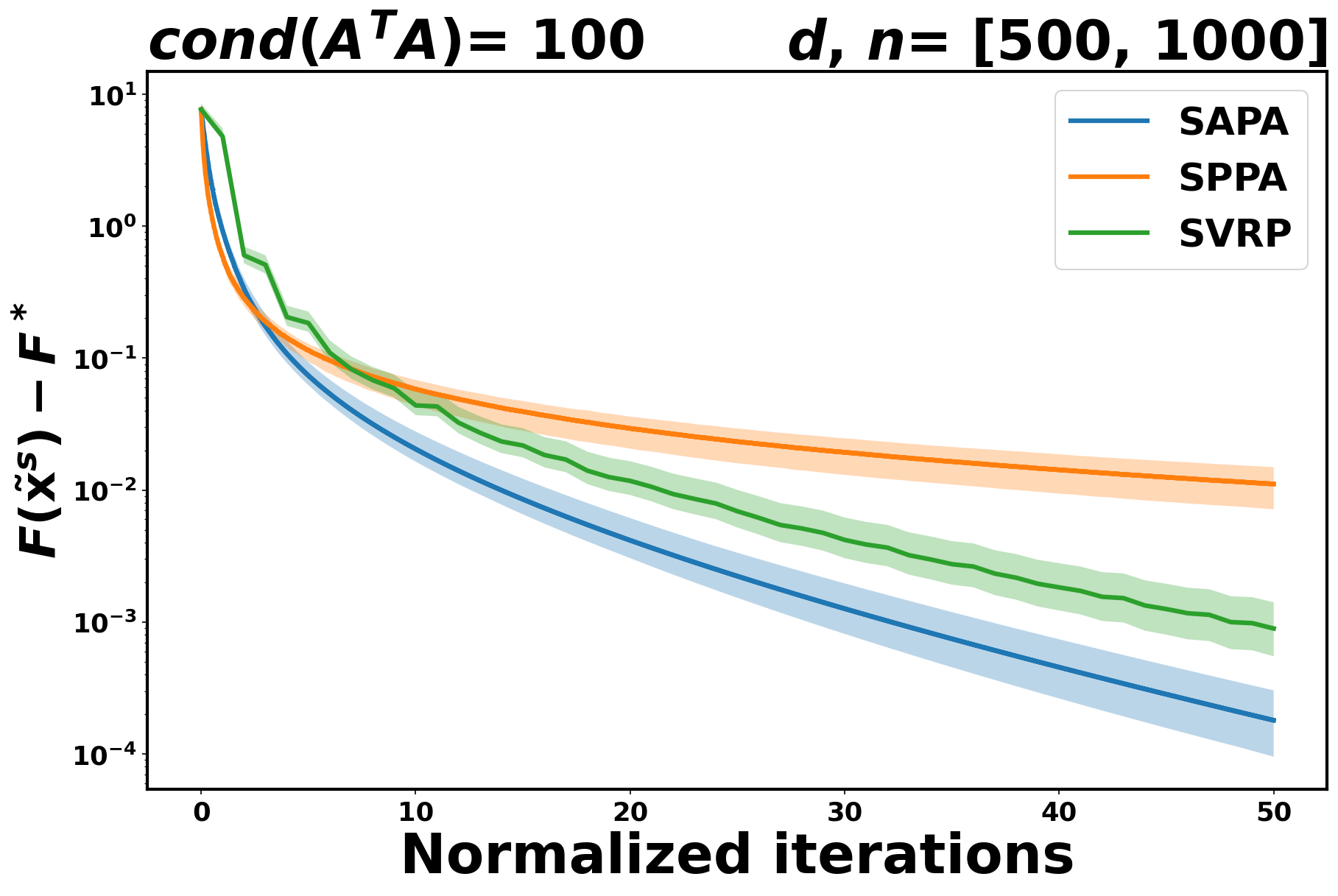}
     \end{subfigure}
     \begin{subfigure}[b]{0.32\textwidth}
         \centering
         \includegraphics[width=\textwidth]
         {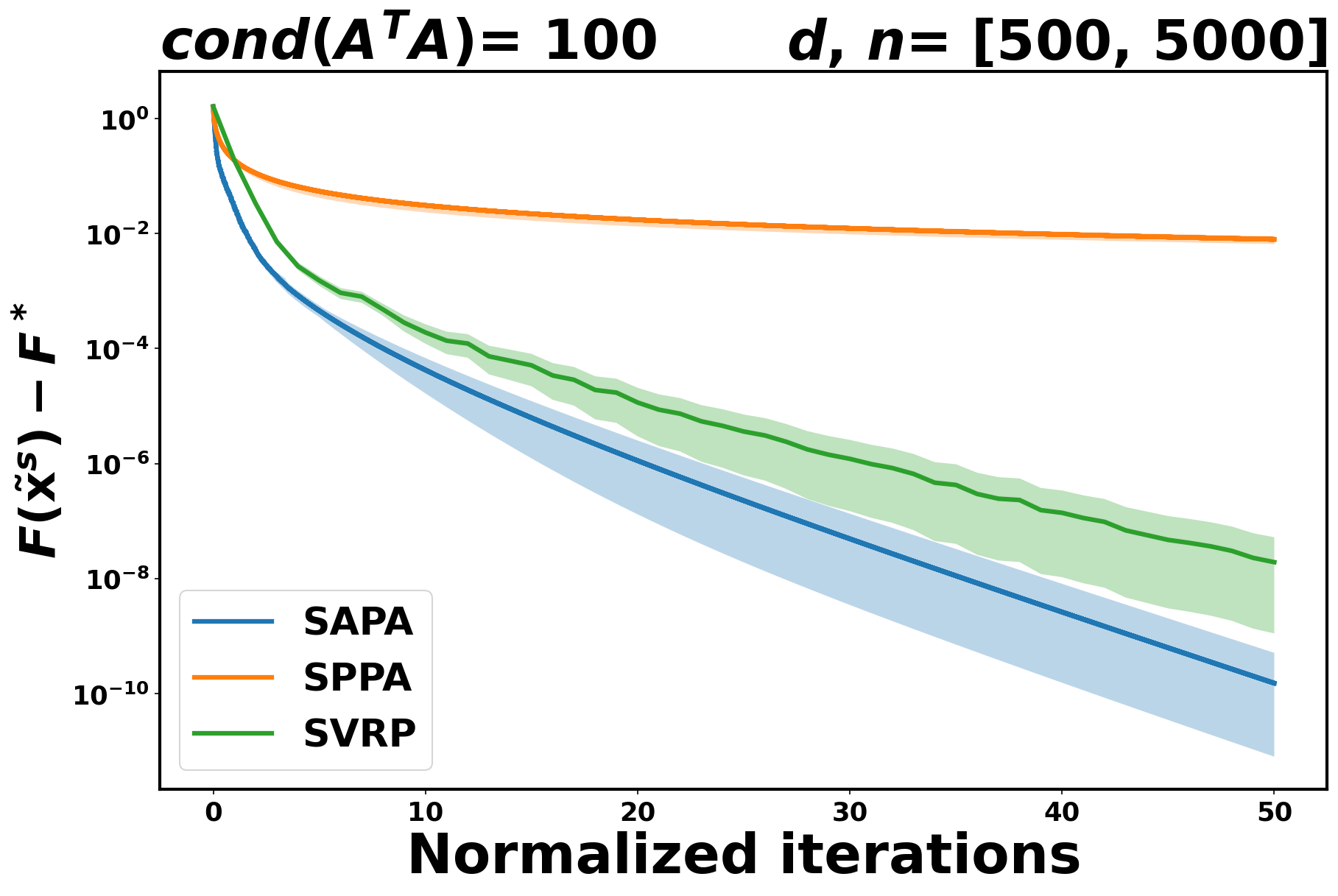}
     \end{subfigure}
    \begin{subfigure}[b]{0.32\textwidth}
         \centering
         \includegraphics[width=\textwidth]
         {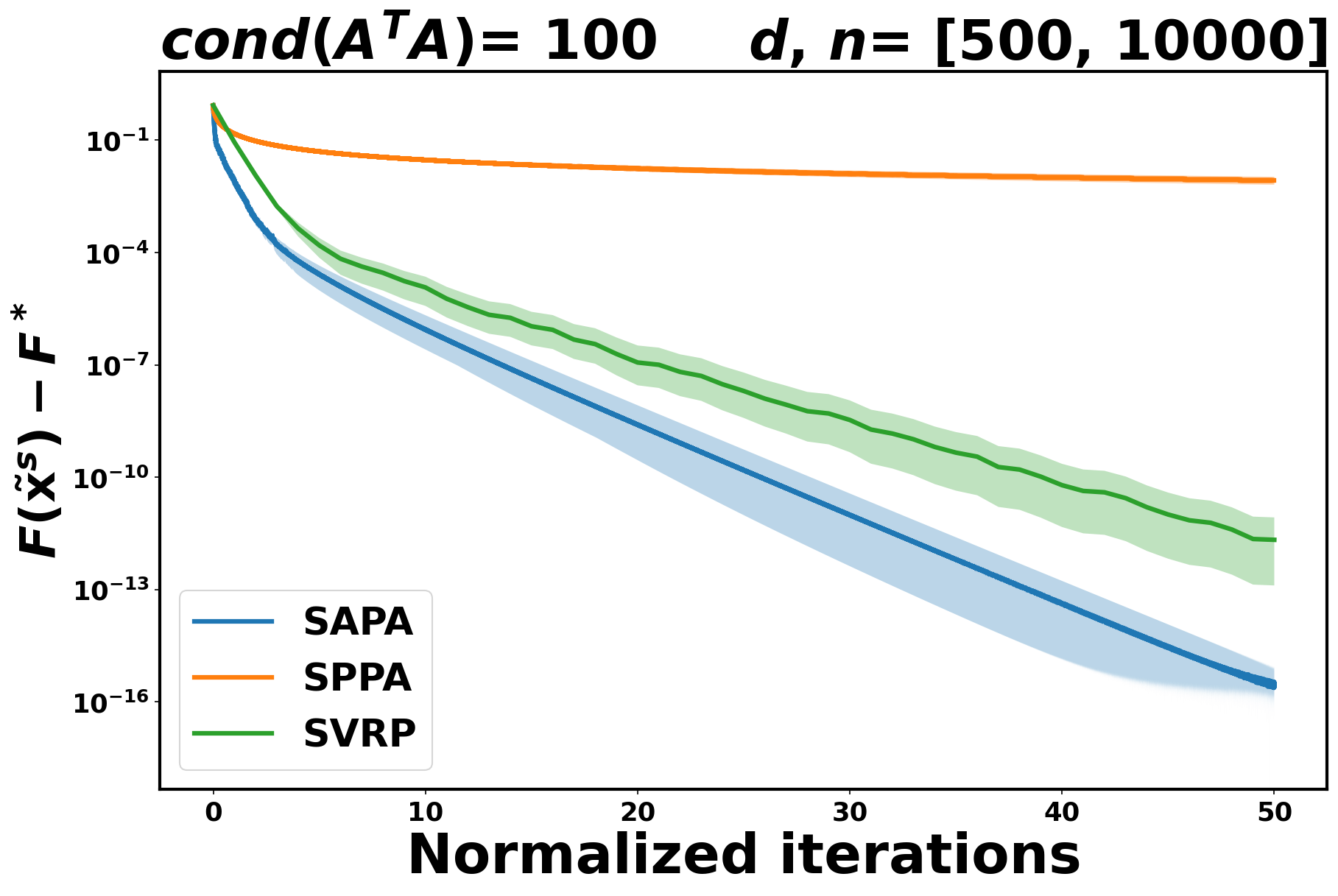}
     \end{subfigure}
    \caption{Evolution of $F(\tilde{\bx}^s)-F_*$, with respect to the normalized iterations counter, for different values of number of functions $n$. We compare the performance of SAPA (blue) and SVRP (green) to SPPA (orange) in the ordinary least squares case.
    }
    \label{fig:olscompsvrpstochprox}
\end{figure}

\subsection{Comparing SAPA to SAGA}
In the next experiments, we implement SAPA and SAGA to solve the least-squares problem \eqref{prob:ols} and logistic problem \eqref{prob:logistic}. The data are exactly as in Section \ref{sect:ols} and the matrix $A$ is generated as explained in Section \ref{sect:logistic}.

The aim of these experiments is twofold: on the one hand, we aim to establish the practical performance of the proposed method in terms of convergence rate and to compare it with the corresponding variance reduction gradient algorithm, on the other hand, we want to assess the stability of SAPA with respect to the step size selection. Indeed, SPPA  has been shown to be more robust and stable with respect to the choice of the step size than the stochastic  gradient algorithms, see \cite{asi2019stochastic, robbins1951stochastic, kim2022convergence}.
In Figure~\ref{fig:compsapasaga} for OLS and Figure~\ref{fig:compsapasagalogistic} for logistic, we plot the number of iterations that are needed for SAPA and SAGA to achieve an accuracy at least $F(x_{t})-F_{\ast} \leq \varepsilon=0.01$, along a fixed range of step sizes. The algorithms run for $n \in \{1000, 5000, 10000\}$, $d$ fixed at $500$ and $cond(A^{\top\!}A)$ at $100$.

Our experimental results show that if the step size is small enough, SAPA and SAGA  behave very similarly for both problems, see Figure \ref{fig:compsapasaga} and Figure \ref{fig:compsapasagalogistic}. This behavior is in agreement with the theoretical results that establish convergence rates for SAPA that are similar to those of SAGA. By increasing the step sizes, we observe that SAPA is more stable than SAGA: the range of step sizes for which SAPA converges is wider than that for SAGA. 

\begin{figure}[ht]
     \centering
     \begin{subfigure}[b]{0.32\textwidth}
         \centering
         \includegraphics[width=\textwidth]{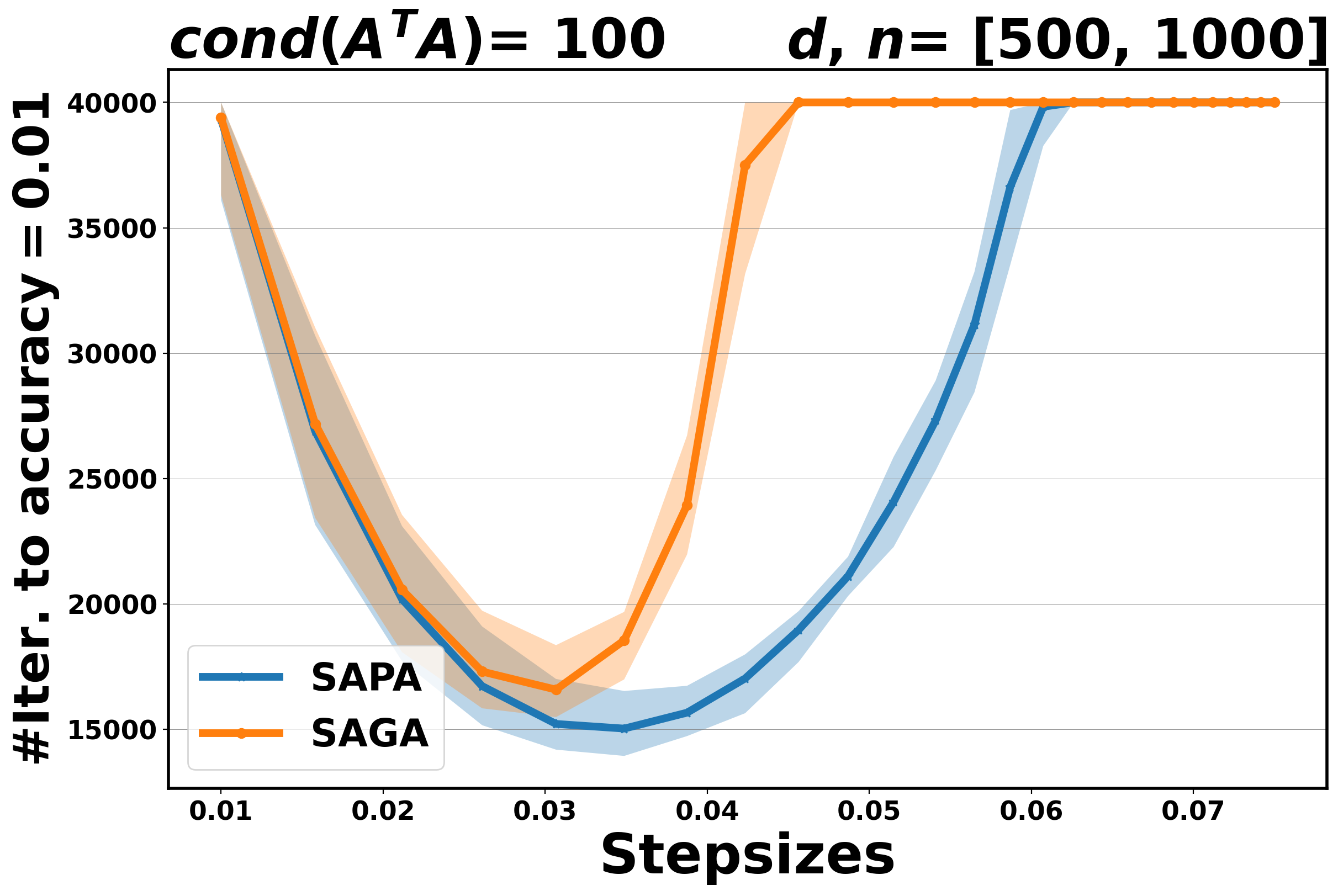}
     \end{subfigure}
     \begin{subfigure}[b]{0.32\textwidth}
         \centering
         \includegraphics[width=\textwidth]{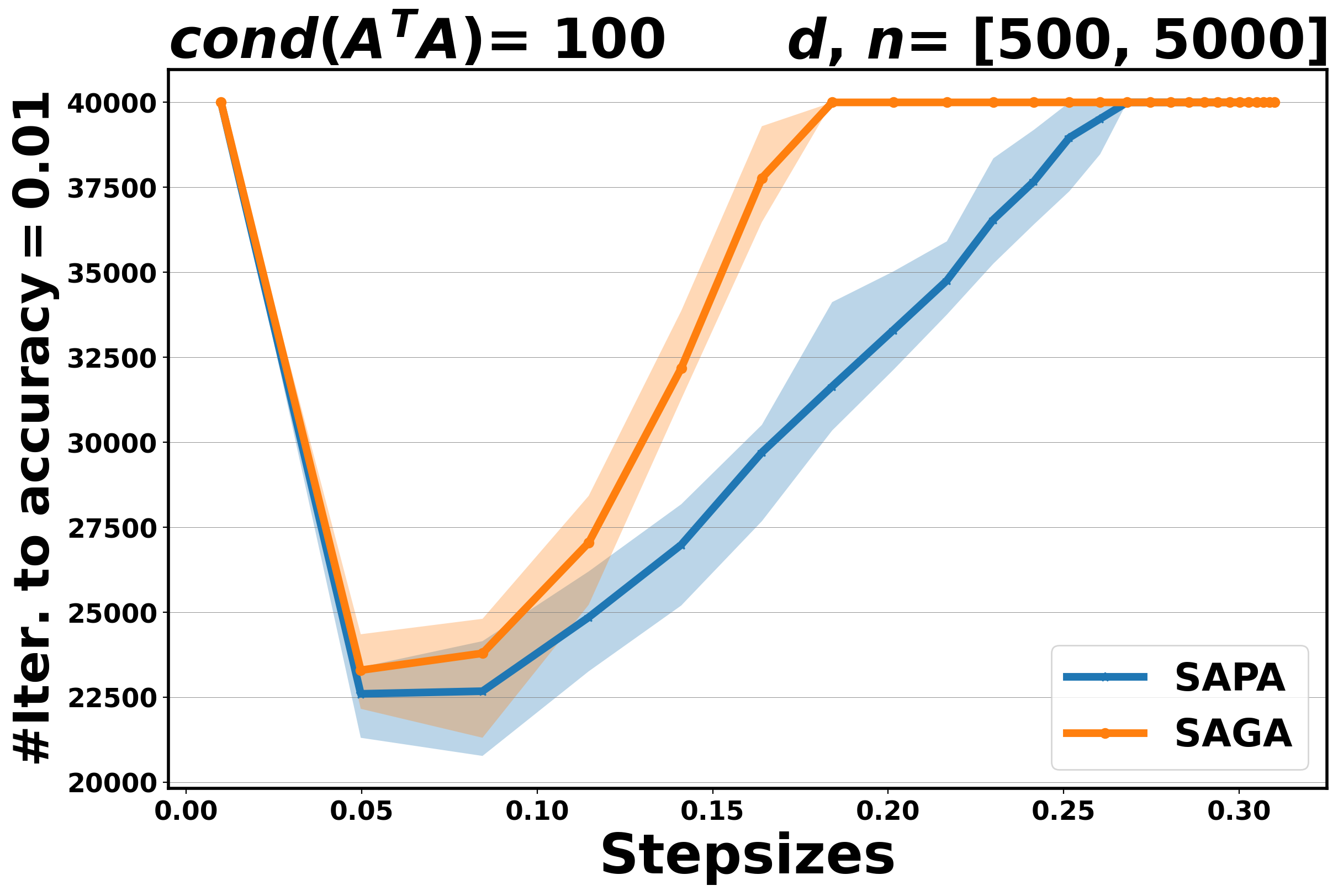}
     \end{subfigure}
    \begin{subfigure}[b]{0.32\textwidth}
         \centering
         \includegraphics[width=\textwidth]{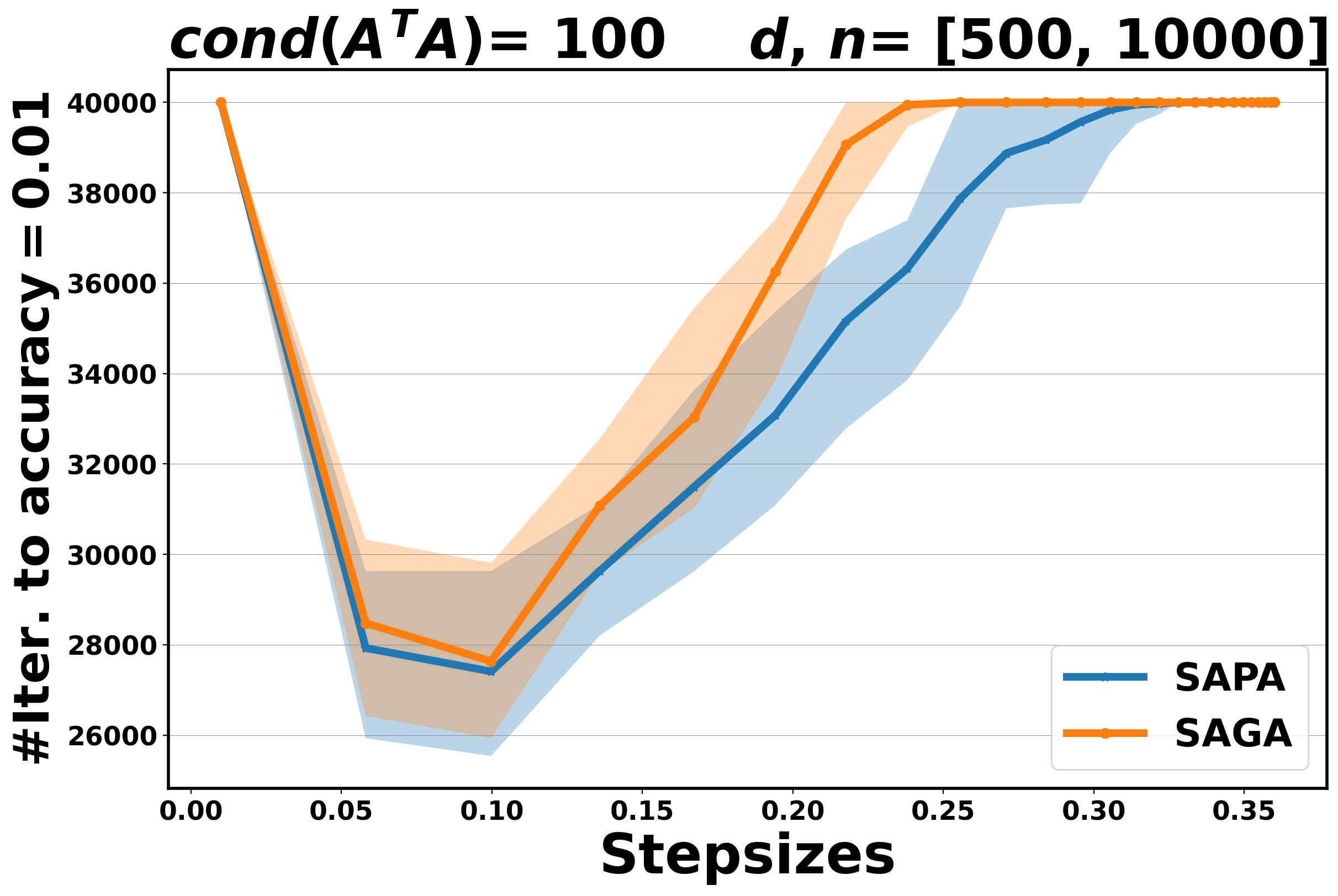}
     \end{subfigure}
    \caption{
    Number of iterations needed in order to achieve an accuracy of at least $0.01$ for different step sizes when solving an OLS problem. A cap is put at 40000 iterations. Here we compare SAPA (in blue) with SAGA (in orange) for different values of the number of functions n.
    }\label{fig:compsapasaga}
\end{figure}

\begin{figure}[ht]
     \centering
     \begin{subfigure}[b]{0.32\textwidth}
         \centering
         \includegraphics[width=\textwidth]{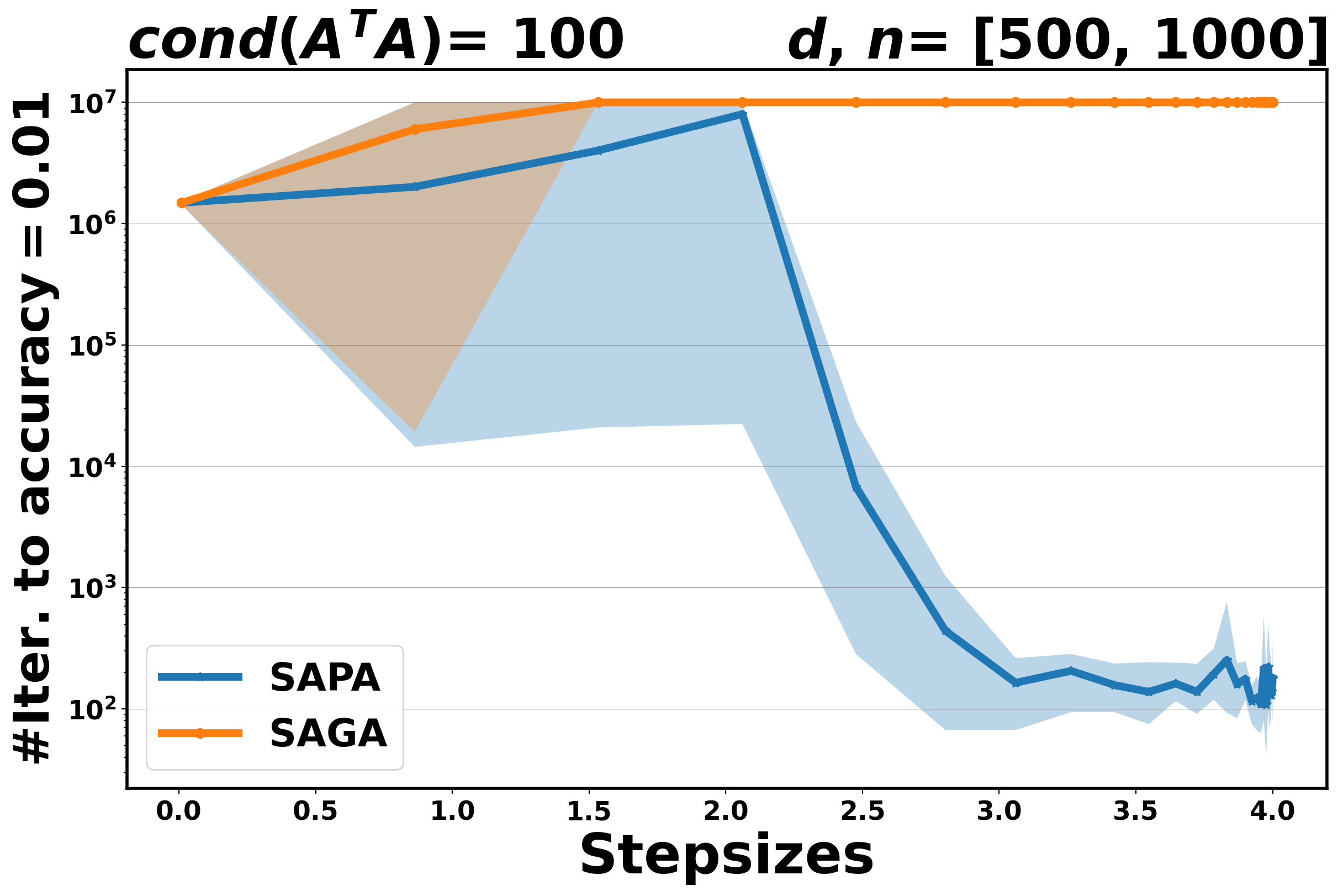}
     \end{subfigure}
     \begin{subfigure}[b]{0.32\textwidth}
         \centering
         \includegraphics[width=\textwidth]{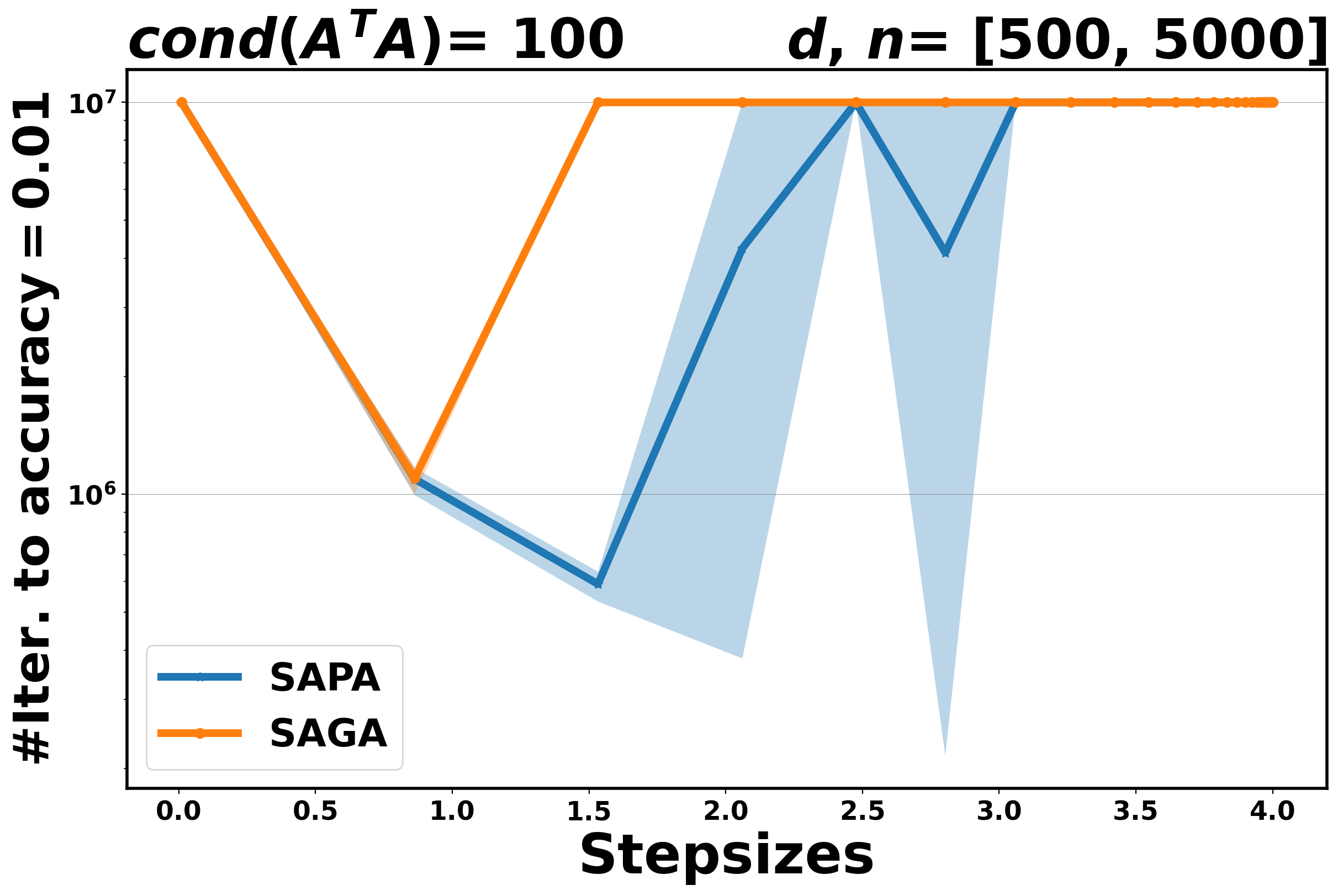}
     \end{subfigure}
    \begin{subfigure}[b]{0.32\textwidth}
         \centering
         \includegraphics[width=\textwidth]{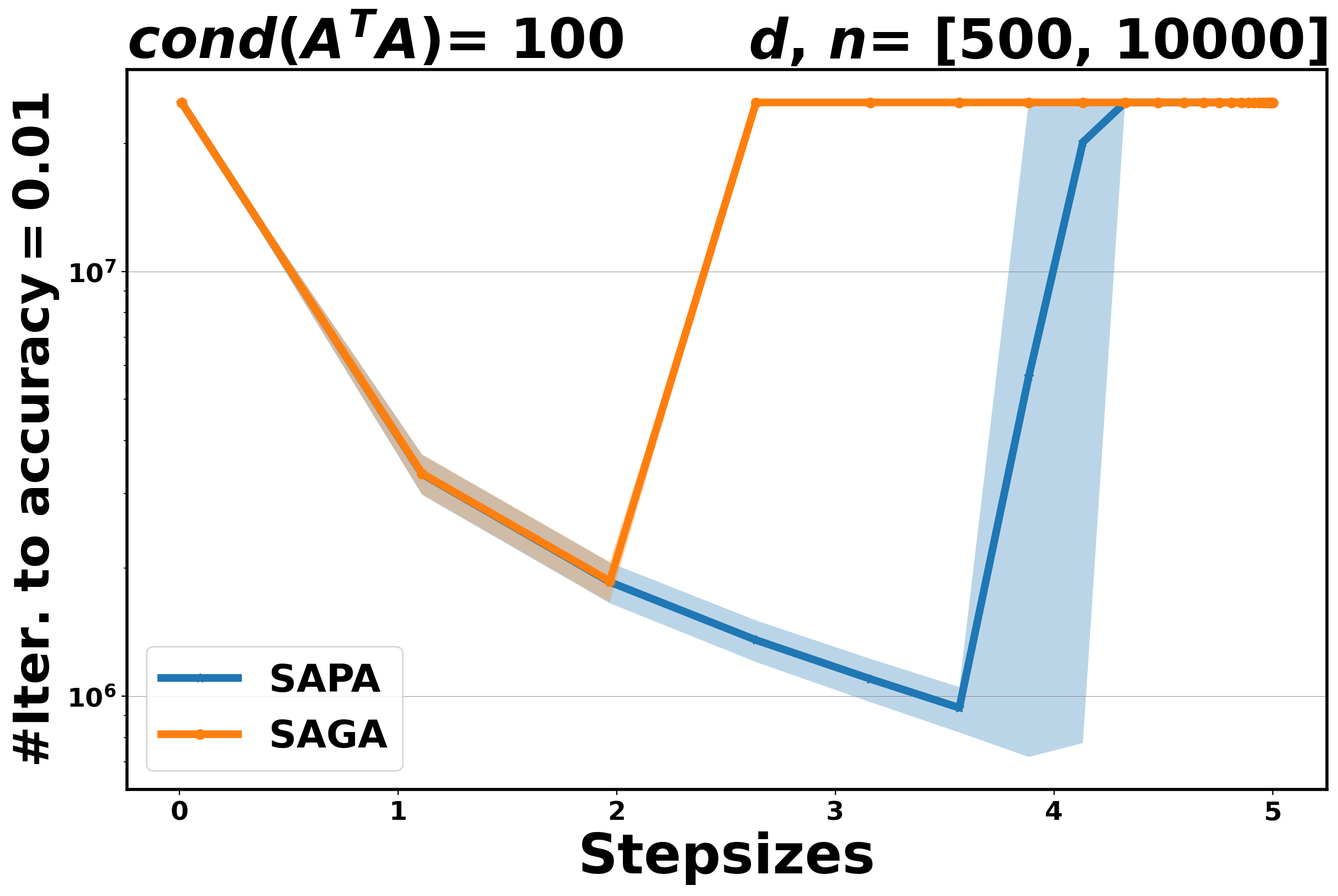}
     \end{subfigure}
    \caption{
    Number of iterations needed in order to achieve an accuracy of at least $0.01$ for different step sizes when solving a logistic problem. A cap is put at $2.5\times 10^{7}$ iterations. Here we compare SAPA (in blue) with SAGA (in orange) for different values of the number of functions n.
    }\label{fig:compsapasagalogistic}
\end{figure}
%%%%%%%%%%%%%%%%%%%%%%%%%%%%%%%%%%%%%%%%%%%%%%%%%%%%%%%%%%%%%%%%%%%%%%%%%%%%%%
\subsection{Comparing SVRP to SVRG}
In this section, we compare the performance of Algorithm \ref{algo:srvprox2} (SVRP) with SVRG \cite{johnson2013accelerating} for the least-squares problem \eqref{prob:ols}, in terms of number of inner iterations (oracle calls) to achieve an accuracy at least $F(x_{t})-F_{\ast} \leq \varepsilon=0.01$, along a fixed range of step sizes. In this experiment, the condition number is set $\text{cond}(A^{\top}A)=100$,  $n=2000$ and the dimension of the objective variable $d$ varies in $\{1000,1500,2000, 3000\}$. The number of outer and inner iterations is set to $s=40$ and $m=1000$, respectively, and the maximum number of iterations (oracle calls) is set to $N=s(m+n+1)$.

% \paragraph{OLS:}
Two main observations can be made. First, we note that when the step size is optimally tuned, SVRP performs always better than SVRG (overall, it requires fewer iterations to achieve $\varepsilon=0.01$ accuracy). Secondly, regarding the stability of the two methods with respect to the choice of the step size, while SVRG seems to be a bit more stable for easy problems ($d=1000$), the situation is reversed for harder problems (e.g., $d\in \{2000,3000\}$), where SVRP is more robust. The results are reported in Figure \ref{fig:compsvrpsvrg}.

\begin{figure}[ht]
     \centering
     \begin{subfigure}[b]{0.44\textwidth}
         \centering
         \includegraphics[width=\textwidth]{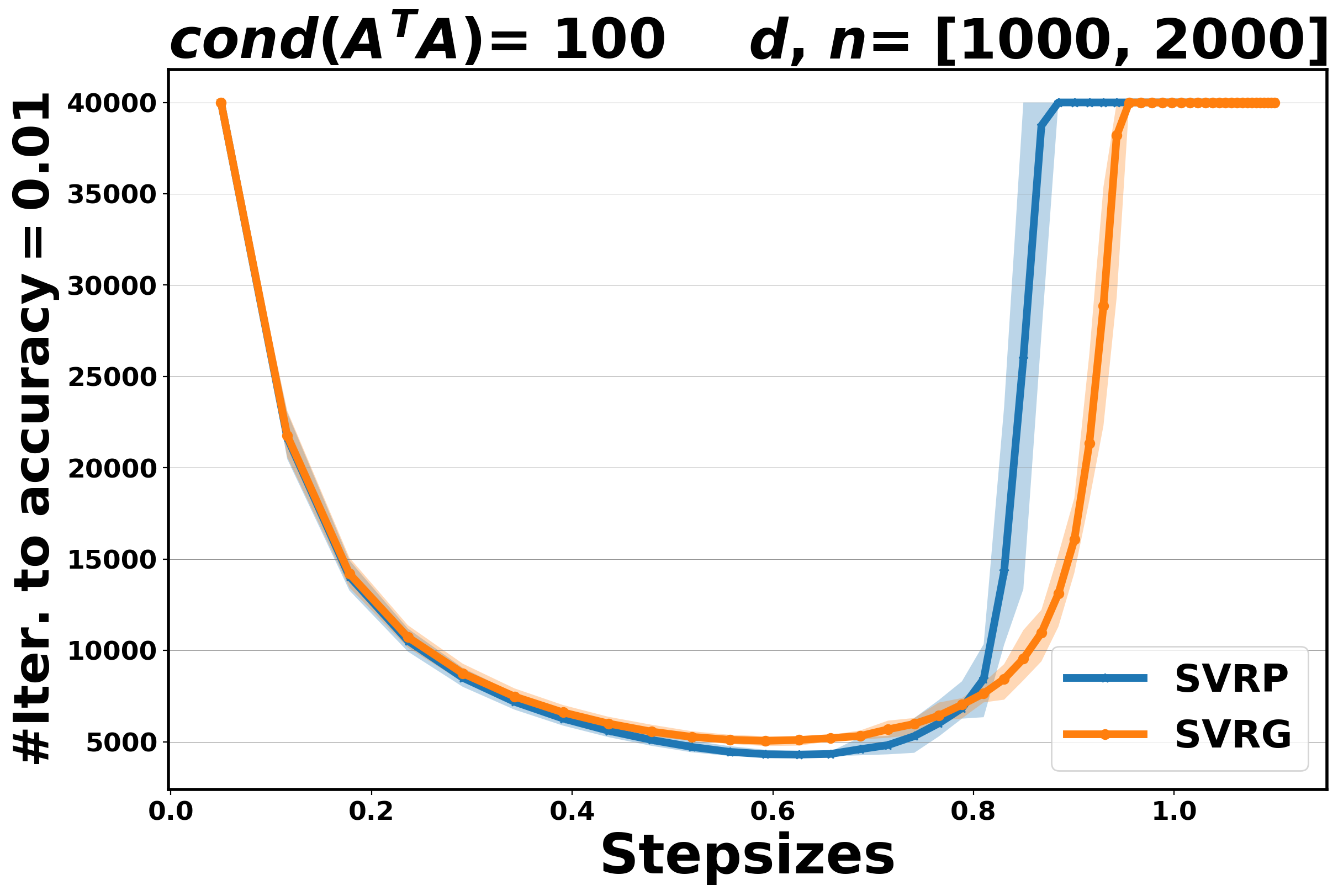}
     \end{subfigure}
          \begin{subfigure}[b]{0.44\textwidth}
         \centering
         \includegraphics[width=\textwidth]{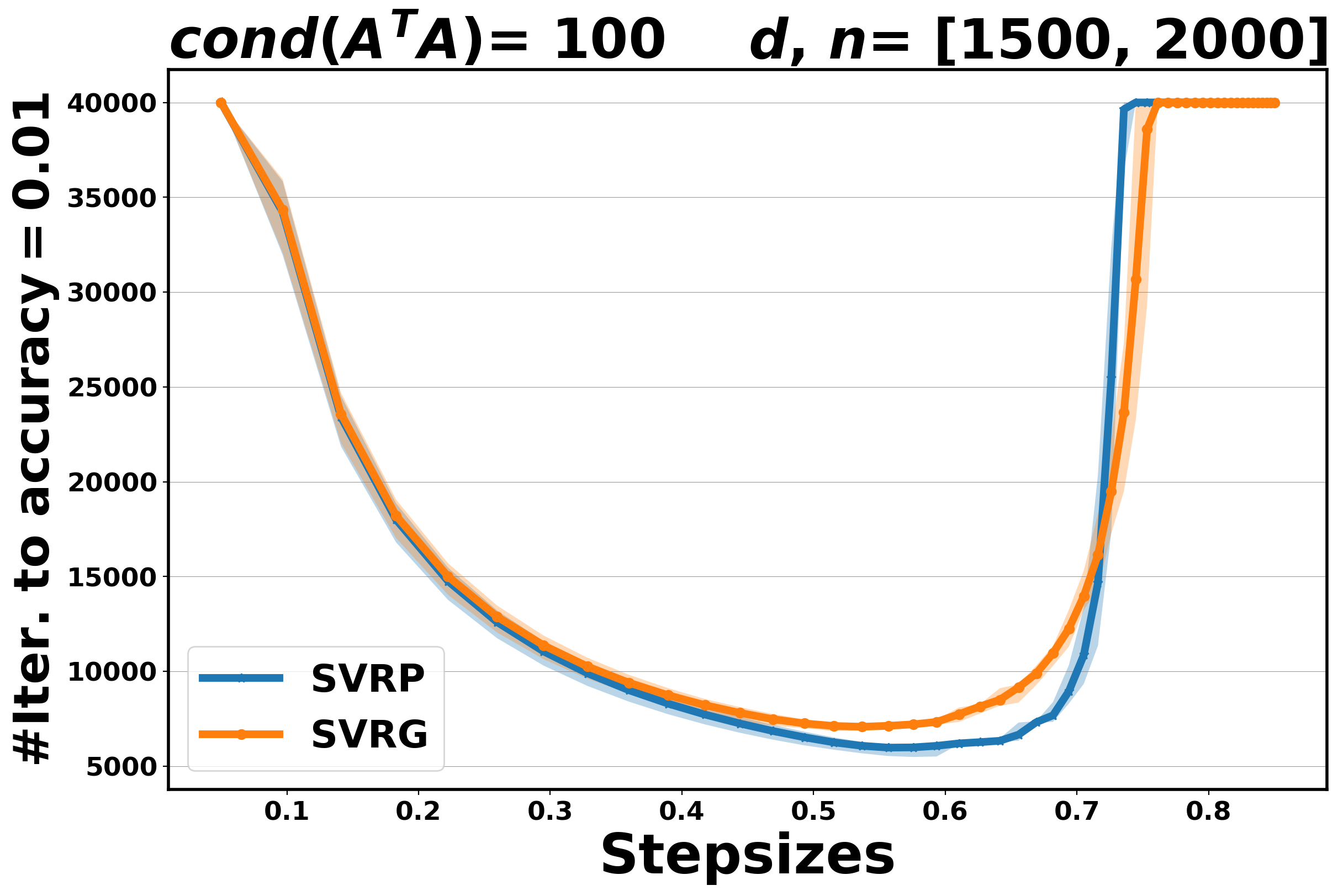}
     \end{subfigure}
     \begin{subfigure}[b]{0.44\textwidth}
         \centering
         \includegraphics[
         % trim={0 0 9mm 0}, 
         width=\textwidth]{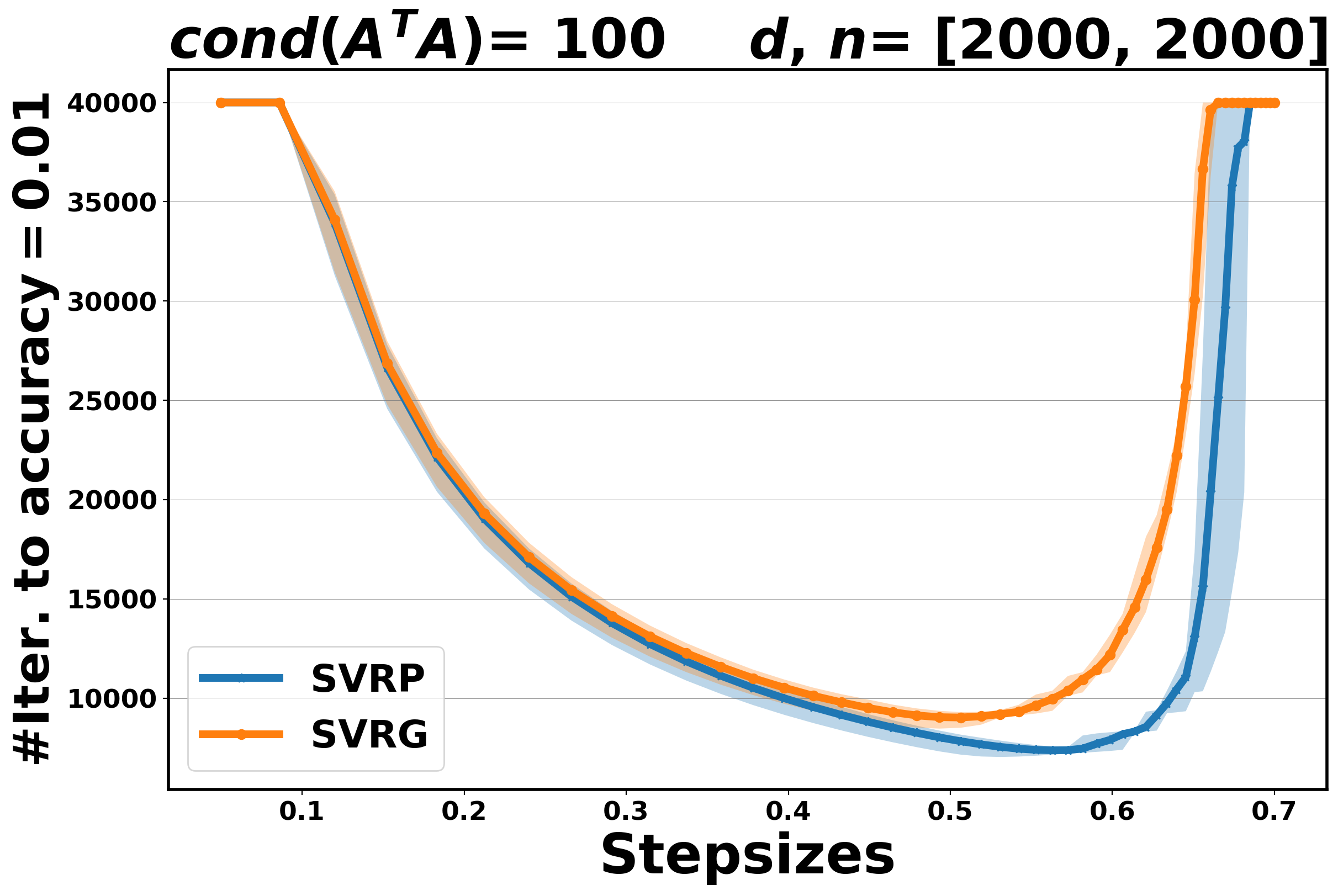}
     \end{subfigure}
     \begin{subfigure}[b]{0.44\textwidth}
         \centering
         \includegraphics[
         % trim={10mm 0 0 0},
         width=\textwidth]{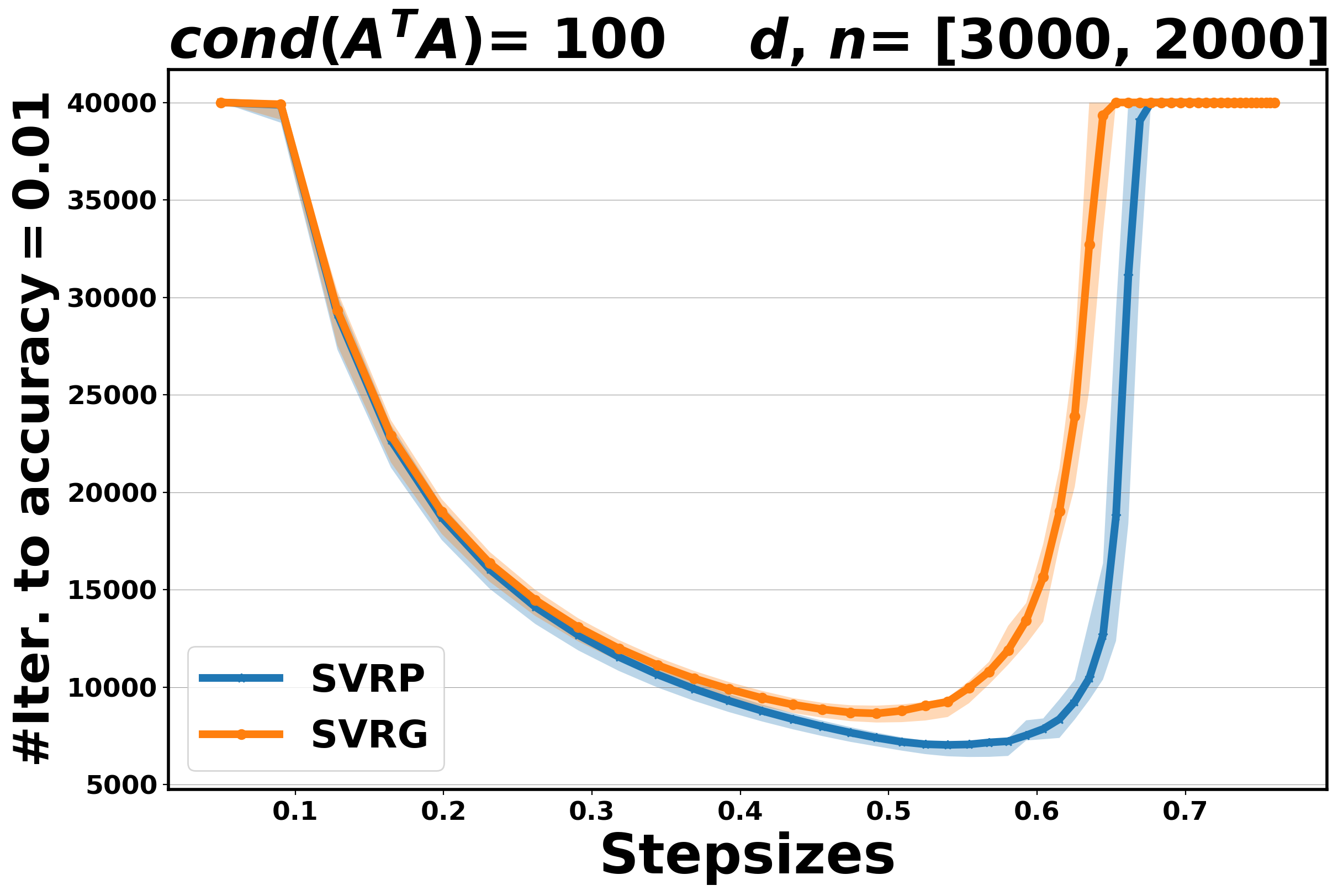}
     \end{subfigure}
    \caption{ Number of iterations needed in order to achieve an accuracy of at least $0.01$ for different step sizes when solving an OLS problem. Here we compare SVRP (in blue) with SVRG (in orange) for four different values of the dimension $d$ in problem \eqref{prob:ols}, starting from $d=1000$ (easy case) to $d=3000$ (hard case).
    }\label{fig:compsvrpsvrg}
\end{figure}

\end{section}
}
%%%%%%%%%%%%%%%%%%%%%%%%%%%%%%%%%%%%%%%%%%%%%%%%%%%%%%%%%%%%%%%%%%%%%%%%%%%%%%%%%%%%%%%%%%%
%%%%%%%%%%%%%%%%%%%%%%%%%%%%%%%%%%%%%%%%%%%%%%%%%%%%%%%%%%%%%%%%%%%%%%%%%%%%%%%%%%%%%%%%%%%%
\begin{section}{Conclusion and future work}\label{sect:conclusion}

In this work, we proposed a general scheme of variance reduction, based on the stochastic proximal point algorithm (SPPA).
In particular, we introduced some new variants of SPPA coupled with various variance reduction strategies,  namely SVRP, SAPA and L-SVRP (see Section \ref{sect:appli}) for which we provide improved convergence results compared to the vanilla version.  As the experiments suggest, the convergence behavior of these stochastic proximal variance-reduced methods is more stable with respect to the step size, rendering them advantageous over their explicit-gradient counterparts.  

Since in this paper, differentiability of the summands is required, an interesting line of future research includes the extension of these results to the nonsmooth case and the model-based framework, as well as the consideration of inertial (accelerated) variants leading, hopefully, to faster rates.
\end{section}
%%%%%%%%%%%%%%%%%%%%%%%%%%%%%%%%%%%%%%%%%%%%%%%%%%%%%%%%%%%%%%%%%%%%%%%%%%%%%%%%%%%%%%%%%%%
% \section*{Data availability}
% The datasets generated during the current study are available from the corresponding author on request and on GitHub\footnote{https://github.com/cheiktraore/Variance-reduction-for-SPPA}.
%%%%%%%%%%%%%%%%%%%%%%%%%%%%%%%%%%%%%%%%%%%%%%%%%%%%%%%%%%%%%%%%%%%%%%%%%%%%%%%%%%%%%%%%%%%%
\appendix
\noindent {\huge \textbf{Appendices}}
\section{Additional proofs} 
In this appendix we provide the proofs of some auxiliary lemmas used in Sections \ref{sect:results} and \ref{sect:appli} in the main core of the paper.
\label{app:proof}
\vspace{2ex}
\subsection{Proofs of Section \texorpdfstring{$\mathbf{3}$}{}}
\label{app:proofsect3}
\noindent
\textbf{Proof of Proposition \ref{prop:unified}.}
Notice that from \eqref{eq:20220915c}, $\bx^{k+1}$ can be identified as
\begin{equation}
    \bx^{k+1}=\argmin_{\bxs \in \HHs}\{\underbrace{f_{i_{k}}(\bx)-\scalarp{\bg^k,\bxs-\bx^{k}}}_{:=R_{i_{k}}(\bxs)}+\frac{1}{2\alpha}\norm{\bxs-\bx^{k}}^2\}. \nonumber
\end{equation}

Let $\bxs \in \HHs$. The function \(R_{i_{k}}(\bxs)=f_{i_{k}}(\bxs)-\scalarp{\bg^k,\bxs-\bx^{k}}\) is convex and continuously differentiable with \(\nabla R_{i_{k}}(\bxs)=\nabla f_{i_{k}}(\bxs)-\bg^k\)

By convexity of $R_{i_{k}}$, we have:
\begin{equation}\label{eq:20220915a}
    \begin{aligned}
    \scalarp{\nabla R_{i_{k}}(\bx^{k+1}) , \bxs-\bx^{k+1}} & \leq R_{i_{k}}(\bxs)-R_{i_{k}}(\bx^{k+1}) \\
    &=f_{i_{k}}(\bxs)-f_{i_{k}}(\bx^{k+1})-\scalarp{\bg^k , \bxs-\bx^{k}} + \scalarp{\bg^k , \bx^{k+1}-\bx^{k}}.
    \end{aligned}
\end{equation}
Since \(\nabla R_{i_{k}}(\bx^{k+1})=\frac{\bx^{k}-\bx^{k+1}}{\alpha}\), from \eqref{eq:20220915a}, it follows:
\begin{equation*}
  \frac{1}{\alpha}\scalarp{\bx^{k}-\bx^{k+1} , \bxs-\bx^{k+1}} \leq f_{i_{k}}(\bxs)-f_{i_{k}}(\bx^{k+1})-\scalarp{\bg^k , \bxs-\bx^{k}} + \scalarp{\bg^k , \bx^{k+1}-\bx^{k}}.
\end{equation*}
By using the identity \(\scalarp{\bx^{k}-\bx^{k+1} , \bxs-\bx^{k+1}}=\frac{1}{2}\norm{\bx^{k+1}-\bx^{k}}^{2} + \frac{1}{2}\norm{\bx^{k+1}-\bxs}^{2} -\frac{1}{2}\norm{\bx^{k}-\bxs}^{2}\) in the previous inequality, we find
\begin{align}
\label{eq:20220915f}
-\scalarp{\bg^k, \bx^{k+1}-\bx^{k}} + &f_{i_{k}}(\bx^{k+1})- f_{i_{k}}(\bxs) + \frac{1}{2\alpha}\norm{\bx^{k+1}-\bx^{k}}^{2} \nonumber\\
        & \leq
         \frac{1}{2\alpha}\norm{\bx^{k}-\bxs}^{2} - \frac{1}{2\alpha}\norm{\bx^{k+1}-\bxs}^{2}  -\scalarp{\bg^k , \bxs-\bx^{k}}.
\end{align}
Recalling the definition of $\bv_k$ in \eqref{defvk}, we can lower bound the left hand term as follows:
\begin{align}
-\scalarp{\bg^k & , \bx^{k+1}-\bx^{k}} + f_{i_{k}}(\bx^{k+1})- f_{i_{k}}(\bxs) + \frac{1}{2\alpha}\norm{\bx^{k+1}-\bx^{k}}^{2} \nonumber\\ 
&= -\scalarp{\bg^k , \bx^{k+1}-\bx^{k}} + f_{i_{k}}(\bx^{k+1})-f_{i_{k}}(\bx^{k})+ \frac{1}{2\alpha}\norm{\bx^{k+1}-\bx^{k}}^{2}\nonumber\\
&\qquad +f_{i_{k}}(\bx^{k})- f_{i_{k}}(\bxs) \nonumber\\
& \geq -\scalarp{\bg^k , \bx^{k+1}-\bx^{k}} + \scalarp{\nabla f_{i_{k}}(\bx^k) ,\bx^{k+1}-\bx^{k}}+ \frac{1}{2\alpha}\norm{\bx^{k+1}-\bx^{k}}^{2}\nonumber\\
&\qquad +f_{i_{k}}(\bx^{k})- f_{i_{k}}(\bxs) \nonumber\\
&= \scalarp{-\bg^k + \nabla f_{i_{k}}(\bx^k) , \bx^{k+1}-\bx^{k}} +  \frac{1}{2\alpha}\norm{\bx^{k+1}-\bx^{k}}^{2}\nonumber\\
&\qquad +f_{i_{k}}(\bx^{k})- f_{i_{k}}(\bxs)\nonumber\\
&= \scalarp{\bv^k , \bx^{k+1}-\bx^{k}} +  \frac{1}{2\alpha}\norm{\bx^{k+1}-\bx^{k}}^{2}\nonumber\\
&\label{eq:20220915d}\qquad +f_{i_{k}}(\bx^{k})- f_{i_{k}}(\bxs),
\end{align}
where in the first inequality, we used the convexity of $f_i$, for all $i \in [n]$. Since
\begin{equation*}
    \scalarp{\bv^k , \bx^{k+1}-\bx^{k}} +  \frac{1}{2\alpha}\norm{\bx^{k+1}-\bx^{k}}^{2} = \left\|\frac{\sqrt{\alpha}}{\sqrt{2}}\bv^k + \frac{1}{\sqrt{2\alpha}}(\bx^{k+1}-\bx^{k})\right\|^2 - \frac{\alpha}{2}\|\bv^k\|^2.
\end{equation*}
From \eqref{eq:20220915d}, it follows
\begin{align}\label{eq:20230215a}
    -\scalarp{\bg^k , \bx^{k+1}-\bx^{k}} + &f_{i_{k}}(\bx^{k+1})- f_{i_{k}}(\bxs) + \frac{1}{2\alpha}\norm{\bx^{k+1}-\bx^{k}}^{2}\nonumber \\
    &\geq\left\|\frac{\sqrt{\alpha}}{\sqrt{2}}\bv^k + \frac{1}{\sqrt{2\alpha}}(\bx^{k+1}-\bx^{k})\right\|^2  - \frac{\alpha}{2}\|\bv^k\|^2 + f_{i_{k}}(\bx^{k})- f_{i_{k}}(\bxs) \nonumber \\
    &= \frac{\alpha}{2} \left\|\nabla f_{i_{k}}(\bx^{k}) - \nabla f_{i_{k}}(\bx^{k+1}) \right\|^2 - \frac{\alpha}{2}\|\bv^k\|^2 + f_{i_{k}}(\bx^{k})- f_{i_{k}}(\bxs).
\end{align}
By using \eqref{eq:20230215a} in \eqref{eq:20220915f}, we obtain
\begin{align}
\label{eq:20220915e}
- \frac{\alpha}{2}\|\bv^k\|^2 + f_{i_{k}}(\bx^{k})- f_{i_{k}}(\bxs) &\leq
         \frac{1}{2\alpha}\norm{\bx^{k}-\bxs}^{2} - \frac{1}{2\alpha}\norm{\bx^{k+1}-\bxs}^{2}  -\scalarp{\bg^k, \bxs-\bx^{k}}.
\end{align}
Now, define $\EE_k[\cdot] = \EE[\cdot \,\vert\, \Fk_k]$,
where $\Fk_k$ is defined in Assumptions~ \ref{ass:variance} and is such that $\bx^k$ is $\Fk_k$-measurable and $i_k$ is independent of $\Fk_k$. Thus, taking the conditional expectation of inequality \ref{eq:20220915e} and rearranging the terms, we have
\begin{align*}
     \EE_k[\norm{\bx^{k+1}-\bxs}^{2}] &\leq \norm{\bx^{k}-\bxs}^{2}  -2 \alpha\EE_{k}\big[f_{i_{k}}(\bx^{k})- f_{i_{k}}(\bxs)\big] + \alpha^2\EE_{k}\|\bv^k\|^2\\
     &= \norm{\bx^{k}-\bxs}^{2} -2 \alpha (F(\bx^{k})- F(\bxs))   + \alpha^2\EE_{k}\|\bv^k\|^2.
\end{align*}
 Replacing $\bxs$ by $\by^{k}$, with $ \by^{k} \in \argmin F$ such that $\norm{\bx^{k}-\by^{k}} = \dist(\bx^{k}, \argmin F)$  and using Assumption \ref{ass:b2}, we get
\begin{align} \label{eq:20230223b}
    \EE_{k}\norm{\bx^{k+1}-\by^{k}}^{2} &\leq \norm{\bx^{k}-\by^{k}}^{2} \nonumber \\
     &\qquad -2 \alpha(F(\bx^{k})- F_*) + \alpha^2\left(2A(F(\bx^k) - F_*) + B\sigma_k^2 
     %- G\be^{k} 
     + D\right)\; a.s.\nonumber \\
     &= \norm{\bx^{k}-\by^{k}}^{2} 
     %- \alpha^2G\be^{k} 
     + \alpha^2 D\nonumber\\
     &\qquad -2 \alpha(1-\alpha A)(F(\bx^{k})- F_*) + \alpha^2 B\sigma_k^2\;\; a.s.
\end{align}
By taking the total expectation in \eqref{eq:20230223b} and using \ref{ass:b3}, for all $M > 0$ we have
\begin{align}%\label{eq:20220916a}
    \EE\norm{\bx^{k+1}-\by^{k}}^{2} + \alpha^2M\EE[\sigma_{k+1}^2] &\leq \EE\norm{\bx^{k}-\by^{k}}^{2} 
    %- \alpha^2G\EE[\be^{k}] 
    + \alpha^2 \EE[D] \nonumber \\
     &\qquad -2 \alpha(1-\alpha A)\EE[F(\bx^{k})- F_*] + \alpha^2 B\EE[\sigma_k^2] \nonumber \\
     &\qquad + \alpha^2M(1-\rho) \EE[\sigma_k^2] + 2\alpha^2MC\EE[F(\bx^k) - F_*] 
     % + \alpha^2 M G \nonumber \\
    \nonumber \\
     &= \EE\norm{\bx^{k}-\by^{k}}^{2} 
     % + \alpha^2 \left( D + M G\right)\nonumber \\
     %- \alpha^2G\EE[\be^{k}] 
     + \alpha^2 \EE[D]\nonumber \\
     &\qquad -2 \alpha\left[1-\alpha (A+ MC)\right]\EE[F(\bx^{k})- F_*] \nonumber \\
     &\qquad + \alpha^2 \left[M+B-\rho M\right]\EE[\sigma_k^2] \nonumber \\
     &= \EE[\dist(\bx^{k}, \argmin F)^2] 
     % + \alpha^2 \left( D + M G\right)\nonumber \\
     %- \alpha^2G\EE[\be^{k}] 
     + \alpha^2 \EE[D]\nonumber \\
     &\qquad -2 \alpha\left[1-\alpha (A+ MC)\right]\EE[F(\bx^{k})- F_*] \nonumber \\
     &\qquad + \alpha^2 \left[M+B-\rho M\right]\EE[\sigma_k^2]. \nonumber
\end{align}
Since $ \EE[\dist(\bx^{k+1}, \argmin F)^2] \leq \EE\norm{\bx^{k+1}-\by^{k}}^{2}$, the statement follows.
\hfill \qed \\
%%%%%%%%%%%%%%%%%%%%%%%%%%%%%%%%%%%%%%%%%%%%%%%%%%%%%%%%%%%%%%%%%%%%%%%%%%%%%%%%%%%%%%%
\vspace{2ex}

\subsection{Proofs of Section \texorpdfstring{$\mathbf{4}$}{}}
\label{app:proofsect4}
\noindent
\textbf{Proof of Lemma \ref{lem:20220927a}.}
Given any $i \in [n]$,
since $f_i$ is convex and $L$-Lipschitz smooth, it is a standard fact (see \cite[Equation 2.1.7]{nesterov2003introductory}) that 
\begin{equation*}
    (\forall\, \bxs, \bys \in \HHs)\quad
\big\|\nabla f_{i}(\bxs)-\nabla f_{i}(\bys)\big\|^{2} \leq 2 L\big[f_{i}(\bxs)-f_{i}(\bys)- \scalarp{\nabla f_{i}(\bys),\bxs-\bys}\big].
\end{equation*}
%the function
%\begin{equation*}
%h_{i}(\bxs)=f_{i}(\bxs)-f_{i}(\by^{k})-\nabla f_{i}(\by^{k})^{\top}(\bxs-\by^{k}).
%\end{equation*}
%We know that $h_{i}\left(\by^{k}\right)=\min _{\bx} h_{i}(\bx)$ since $h_i$ is convex and $\nabla h_{i}\left(\by^{k}\right)=0$. Therefore
%\begin{align*}
%0=h_{i}(\by^{k}) & \leq \min _{\alpha}\left[h_{i}\left(\bx-\alpha \nabla h_{i}(\bx)\right)\right] \\
%& \leq \min _{\alpha}\left[h_{i}(\bx)-\alpha\left\|\nabla h_{i}(\bx)\right\|^{2}+ \frac{L}{2} \alpha^{2}\left\|\nabla h_{i}(\bx)\right\|^{2}\right]=h_{i}(\bx)-\frac{1}{2 L}\left\|\nabla h_{i}(\bx)\right\|^{2} .
%\end{align*}
%The second inequality comes from the fact that $h_i$ is $L$-smooth. It follows 
%\begin{equation*}
%    \big\|\nabla f_{i}(\bxs)-\nabla f_{i}(\by^{k})\big\|^{2} \leq 2 L\big[f_{i}(\bxs)-f_{i}(\by^{k})-\nabla f_{i}(\by^{k})^{\top}(\bxs-\by^{k})\big].
%\end{equation*}
By summing the above inequality over $i=1, \ldots, n$, 
we have
\begin{equation*}
\sum_{i=1}^{n} \frac{1}{n} \big\|\nabla f_{i}(\bxs)-\nabla f_{i}(\bys)\big\|^{2} \leq 2 L\big[F(\bxs)-F(\bys)\big] -2L \scalarp{\nabla F(\bys), \bxs - \bys}
\end{equation*}
and hence if we suppose that $\nabla F\left(\bys\right)=0$, then we obtain
\begin{equation}
\label{eq:20220927b}
\sum_{i=1}^{n} \frac{1}{n} \big\|\nabla f_{i}(\bxs)-\nabla f_{i}(\bys)\big\|^{2} \leq 2 L\big[F(\bxs)-F(\bys)\big] .
\end{equation}
Now, let $s \in \N$ and $k \in \{0,\dots, m-1\}$
and set, for the sake of brevity, $j_k = i_{sm+k}$.
Defining $\Fk_{k} = \Fk_{s,k} = \sigma(\xi_0, \dots, \xi_{s-1}, i_0, \dots, i_{sm+k-1})$ and
$\EE_k[\cdot] = \EE[\cdot \,\vert\, \Fk_k]$, and recalling that
\begin{equation*}
    \bv^k = \nabla f_{j_k}(\bx^k) - \nabla f_{j_k}(\tilde{\bx}^s) + \nabla F(\tilde{\bx}^s)
    \quad\text{and}\quad
    \tilde{\by}^s = P_{\argmin F}(\tilde{\bx}^s),
\end{equation*}
we obtain
\begin{align}
% \label{eq:20220927a}
\EE_{k} \big[\| \bv^{k}\|^{2}\big] &\leq  2 \EE_{k}\big\|\nabla f_{j_{k}}(\bx^{k})-\nabla f_{j_{k}}(\tilde{\by}^{s})\big\|^{2}+2 \EE_{k}\big\|[\nabla f_{j_{k}}(\tilde{\bx}^{s})-\nabla f_{j_{k}}(\tilde{\by}^{s})]-\nabla F(\tilde{\bx}^{s})\big\|^{2} \nonumber \\[1ex]
&= 2 \EE_{k} \big\|[\nabla f_{j_{k}}(\tilde{\bx}^{s})-\nabla f_{j_{k}}(\tilde{\by}^{s})\big] -\EE_{k}[\nabla f_{j_{k}}(\tilde{\bx}^{s})-\nabla f_{j_{k}}(\tilde{\by}^{s})]\big \|^{2} \nonumber \\[1ex]
& \quad + 2 \EE_{k}\big\|\nabla f_{j_{k}}(\bx^{k})-\nabla f_{j_{k}}(\tilde{\by}^{s})\big\|^{2}  \nonumber \\[1ex]
&= 2 \EE_{k}\big\|\nabla f_{j_{k}}(\tilde{\bx}^{s})-\nabla f_{j_{k}}(\tilde{\by}^{s})\big\|^{2} - 2 \big\|\EE_{k}[\nabla f_{j_{k}}(\tilde{\bx}^{s})-\nabla f_{j_{k}}(\tilde{\by}^{s})]\big\|^2 \nonumber\\[1ex]
& \qquad + 2 \EE_{k}\big\|\nabla f_{j_{k}}(\bx^{k})-\nabla f_{j_{k}}(\tilde{\by}^{s})\big\|^{2} \nonumber \\[1ex]
&= 2 \EE_{k}\big\|\nabla f_{j_{k}}(\bx^{k})-\nabla f_{j_{k}}(\tilde{\by}^{s})\big\|^{2}+2 \EE_{k}\big\|\nabla f_{j_{k}}(\tilde{\bx}^{s})-\nabla f_{j_{k}}(\tilde{\by}^{s})\big\|^{2} \nonumber \\[1ex]
& \qquad -2\|\nabla F(\tilde{\bx}^{s})\|^2 \nonumber \\
&\leq 4 L\big( F(\bx^{k})-F(\tilde{\by}^{s})\big) + 2\sigma_k^2 -2\|\nabla F(\tilde{\bx}^{s})\|^2, \nonumber
\end{align}
where in the last inequality, we used \eqref{eq:20220927b} with $\bys=\tilde{\by}^s$.
\hfill \qed \\
%%%%%%%%%%%%%%%%%%%%%%%%%%%%%%%%%%%%%%%%%%%%%%%%%%%%%%%%%%%%%%%%%%%%%%%%%%%%%%%%%%%%%%%
% \vspace{2ex}

% \noindent
% \textbf{Proof of Corollary \ref{cor:unifiedsvrg}}
% Since Assumption \ref{ass:a11} is true, by Lemma \ref{lem:20220927a} the Assumptions \ref{ass:variance} are verified with $\bv_{k}=\nabla f_{i_k}\left(\bx^{k}\right), A = 2L, B=2, \rho=0,$ and $C = 0$. So the result is just Proposition \ref{prop:unified} with the right constant.
% \hfill \qed \\
%%%%%%%%%%%%%%%%%%%%%%%%%%%%%%%%%%%%%%%%%%%%%%%%%%%%%%%%%%%%%%%%%%%%%%%%%%%%%%%%%%%%%%%
\vspace{2ex}

\noindent{\textbf{Proof of Lemma \ref{lem:20221103a}} The proof of equation \eqref{eq:20221103a}  is equal to that of \eqref{eq:20220927c} in Lemma \ref{lem:20220927a} using $\bu^k$ instead of $\tilde{\bx}^s$ 
%\vas{(I'm not sure but maybe we want to replace $\tilde{\by}^s$ with $\by_{k}$). Can you recheck?} 
and $\EE_k[\cdot] = \EE[\cdot\,\vert\, \Fk_k]$ with $\Fk_k = \sigma(i_0, \dots, i_{k-1}, \varepsilon^0, \dots, \varepsilon^{k-1})$, which ensures that $\bx^k$ and $\bu^k$ are $\Fk_k$-measurables, and $i_k$ and $\varepsilon^k$ are independent of $\Fk_k$.}
Concerning \eqref{eq:20221103b}, we note that
\begin{equation*}
    \sigma_{k+1}^2 = \frac{1}{n} \sum_{i=1}^n \big\lVert \nabla f_i((1-\varepsilon^k)\bu^k + \varepsilon^k \bx^k) -
\nabla f_i (P_{\argmin F} ((1-\varepsilon^k)\bu^k + \varepsilon^k \bx^k))\big\rVert^2.
\end{equation*}
Therefore, we have
\begin{align*}
\EE_{k}\left[\sigma^2_{k+1}\right] 
& = (1-p) \sigma^2_{k}+p \frac{1}{ n} \sum_{i=1}^{n}\big\|\nabla f_i(\bx^{k})-\nabla f_i(P_{\argmin F}(\bx^k))\big\|^{2} \\
& \leq (1-p) \sigma^2_{k}+ 2 p L (F(\bx^{k})-F^{*}),
\end{align*}
where in the last inequality we used \eqref{eq:20220927b} with $\bys = P_{\argmin F}(\bx^k)$.
\hfill \qed \\
%%%%%%%%%%%%%%%%%%%%%%%%%%%%%%%%%%%%%%%%%%%%%%%%%%%%%%%%%%%%%%%%%%%%%%%%%%%%%%%%%%%%%%%
\vspace{2ex}

\noindent
\textbf{Proof of Lemma \ref{lem:20221004a}}
%$L$-smoothness and convexity of $f_i$ implies
%\begin{equation}
%\frac{1}{2 L}\left\|\nabla f_i(x)-\nabla f_i(y)\right\|^2 \leq %f_i(x)-f_i(y)-\left\langle\nabla f_i(y), x-y\right\rangle, \quad \forall x, y \in \mathbb{R}^d, i \in[n]
%\end{equation}
We recall that, by definition,
\begin{equation*}
(\forall\, i \in [n])\quad    \bphi^{k+1}_i =  \bphi^k_i + \delta_{i,i_k}(\bx^k - \bphi^k_i).
\end{equation*}
Now, set $\EE_k [\cdot] =\EE[\cdot\vert\, \Fk_k]$ with $\Fk_k= \sigma(i_0, \dots, i_{k-1})$, so that $\bphi^k_i$ and $\bx^k$ are $\Fk_k$-measurable and $i_k$ is independent of $\Fk_k$.
By definition of $\bv^k$, and using the fact that $\nabla F(\bx^{*})=0$ and inequality \eqref{eq:20220927b}, we have
\begin{align}
\EE_{k}\big[\|\bv^k\|^2\big]&=\EE_{k}\left[\Big\|\nabla f_{i_k}(\bx^k)-\nabla f_{i_k}(\bphi_{i_k}^k)+\frac{1}{n} \sum_{i=1}^n \nabla f_i(\bphi_i^k)-\nabla F(\bx^{*})\Big\|^2\right] \nonumber \\
&=\EE_{k}\bigg[\Big\|\nabla f_{i_k}(\bx^k)-\nabla f_{i_k}(\bx^{*})+\nabla f_{i_k}(\bx^{*})-\nabla f_{i_k}(\bphi_{i_k}^k)  \nonumber\\
&\qquad\qquad +\frac{1}{n} \sum_{i=1}^n \nabla f_i(\bphi_i^k)-\nabla F(\bx^{*})\Big\|^2\bigg] \nonumber \\
&\leq  2 \EE_{k}\left[\big\|\nabla f_{i_k}(\bx^k)-\nabla f_{i_k}(\bx^{*})\big\|^2\right] \nonumber \\
&\qquad\qquad +2 \EE_{k}\left[\big\|\nabla f_{i_k}(\bx^{*})-\nabla f_{i_k}(\bphi_{i_k}^k)-\EE_{k}[\nabla f_{i_k}(\bx^{*})-\nabla f_{i_k}(\bphi_{i_k}^k)]\big\|^2\right] \nonumber \\
&= \frac{2}{n} \sum_{i=1}^n \big\|\nabla f_{i}(\bx^k)-\nabla f_{i}(\bx^{*})\big\|^2+2 \EE_{k}\left[\big\|\nabla f_{i_k}(\bx^{*})-\nabla f_{i_k}(\bphi_{i_k}^k)\big\|^2\right] \nonumber \\
&\qquad\qquad -2 \bigg\|\frac{1}{n} \sum_{i=1}^n \nabla f_i (\bphi_i^k)\bigg\|^2 \nonumber \\
&\leq 4 L \big(F(\bx^k) - F_*\big) +2 \underbrace{\frac{1}{n} \sum_{i=1}^n\big\|\nabla f_i(\bphi_i^k)-\nabla f_i(\bx^{*})\big\|^2}_{\sigma_k^2} -2 \bigg\|\frac{1}{n} \sum_{i=1}^n \nabla f_i (\bphi_i^k)\bigg\|^2. \nonumber
\end{align}
For the second part \eqref{eq:20221004c}, we proceed as follows
\begin{align}
\EE_k\left[\sigma_{k+1}^2\right] &=  \frac{1}{n} \sum_{i=1}^n \EE_k\left[\big\|\nabla f_i(\bphi_i^{k+1})-\nabla f_i(\bx^{*})\big\|^2\right] \nonumber \\
&=  \frac{1}{n} \sum_{i=1}^n \EE_k\left[\big\|\nabla f_i(\bphi_i^{k}+\delta_{i,i_k}(\bx^k-\bphi^k_i))-\nabla f_i(\bx^{*})\big\|^2\right] \nonumber \\
&=  \frac{1}{n} \sum_{i=1}^n \bigg(\frac 1 n \sum_{j=1}^n \big\|\nabla f_i(\bphi_i^k+\delta_{i,j}(\bx^k-\bphi^k_i))-\nabla f_i(\bx^{*})\big\|^2\bigg) \nonumber \\
&= \frac{1}{n} \sum_{i=1}^n\left(\frac{n-1}{n}\big\|\nabla f_i(\bphi_i^k)-\nabla f_i(\bx^{*})\big\|^2+\frac{1}{n}\big\|\nabla f_i(\bx^k)-\nabla f_i(\bx^{*})\big\|^2\right) \nonumber \\
&=\left(1-\frac{1}{n}\right) \frac{1}{n} \sum_{i=1}^n\big\|\nabla f_i(\bphi_i^k)-\nabla f_i(\bx^{*})\big\|^2 +\frac{1}{n^2} \sum_{i=1}^n \big\|\nabla f_i(\bx^k)-\nabla f_i(\bx^{*})\big\|^2 \nonumber \\
&\leq\left(1-\frac{1}{n}\right) \sigma_k^2+\frac{2 L}{n} \big(F(\bx^k) - F_*\big),\nonumber
\end{align} 
where in the last inequality we used inequality \eqref{eq:20220927b} with $\bys = \bx^*$.
\hfill \qed \\

%%%%%%%%%%%%%%%%%%%%%%%%%%%%%%%%%%%%%%%%%%%%%%%%%%%%%%%%%%%%%%%%%%%%%%%%%%%%%%%%%%%%%%%%%%%%

\bibliographystyle{abbrv} %{abbrv} %{amsalpha} {plainnat}
\bibliography{paper}
\end{document}